\documentclass[12pt,reqno]{article}
\usepackage[utf8]{inputenc}
\usepackage{amsmath}
\usepackage{amsfonts}
\usepackage{amssymb}
\usepackage{amsthm}	
\usepackage{cases}
\usepackage{subeqnarray}
\usepackage{cite}
\usepackage[colorlinks]{hyperref}
\usepackage{graphics}
\usepackage{graphicx}
\usepackage{fancyref}
\usepackage{epstopdf}
\usepackage{float}
\usepackage{enumerate}
\usepackage{bm}
\usepackage{titlesec}
\usepackage[titletoc,toc,title]{appendix}

\makeatletter
\allowdisplaybreaks \numberwithin{equation}{section}
\newtheorem{theorem}{Theorem}[section]
\newtheorem{corollary}[theorem]{Corollary}
\newtheorem{lemma}[theorem]{Lemma}

\newtheorem{remark}{Remark}
\allowdisplaybreaks \numberwithin{remark}{section}

\begin{document}

\title{ Heteroclinic Cycles of a Symmetric May-Leonard Competition Model with Seasonal Succession
\thanks{This research is partially supported by NSF of China No.11871231 and Fujian Province University Key Laboratory of
Computational Science, School of Mathematical Sciences, Huaqiao University, Quanzhou, China}
\date{\empty}
\date{\empty}
\author{ {Xizhuang Xie$^{a,b}$, Lin Niu$^b$\thanks{The corresponding author, \emph {E-mail address}: niulin@ustc.edu.cn} } \\
\\[3pt]
{ \small $^{a}$School of Mathematical Sciences, Huaqiao University}\\
{\small Quanzhou, Fujian, 362021, P. R. China}\\
	{\small $^{b}$School of Mathematical Sciences}\\
	{ \small University of Science and Technology of China}\\
{ \small Hefei, Anhui, 230026, P. R. China}\\
}
}\maketitle
\begin{abstract}
	In this paper, we are concerned with the stability of heteroclinic cycles of
 the symmetric May-Leonard competition model with seasonal succession.
Sufficient conditions for stability of  heteroclinic cycles are obtained.
Meanwhile, we present the explicit expression of the carrying simplex in a special case.
By taking $\varphi$ as the switching strategy parameter for a given example,
the bifurcation of stability of the heteroclinic cycle is investigated.
We also find that there are rich dynamics (including
nontrivial periodic solutions, invariant closed curves and heteroclinic cycles) for such a new system via numerical simulation. Biologically, the result is interesting as a caricature of the
complexities that seasonal succession can introduce into the classical May-Leonard competitive model.

	\par
	\textbf{Keywords}: May-Leonard competitive model, Seasonal succession, Heteroclinic Cycles, Carrying simplex, Bifurcation

    \par
	\textbf{AMS Subject Classification (2020)}: 34C12, 34C25, 34D23, 37C65, 92D25
\end{abstract}

\section{Introduction}
Lotka-Volterra competition is modeled by a system of
differential equations describing the competition between two or
more species that compete for the same resources or habitat. There
have been extensive studies and applications of Lotka-Volterra
competition models (see \cite{C1980,SX2012,JN22017,K1974,M2002,M2007,HL1995,YX2013}). In particular, May and Leonard \cite{ML1975}
focused on the case that all three competing species have the same intrinsic growth rates
and the three species compete in the ``rock-scissors-paper" manner, which is modeled by the following form
\begin{equation}
\begin{cases}
\label{ML}
\frac{dx_1}{dt}=x_1(1-x_1-\alpha x_2-\beta x_3),\\
\frac{dx_2}{dt}=x_2(1-\beta x_1-x_2-\alpha x_3),\\
\frac{dx_3}{dt}=x_3(1-\alpha x_1-\beta x_2-x_3),\\
 (x_1(0),x_2(0),x_3(0))=x^0 \in R_+^{3},
\end{cases}
\end{equation}
where $0<\alpha<1<\beta$ and $\alpha+\beta>2$. They found numerically that system (\ref{ML}) exhibits
a general class of solutions with non-periodic oscillations of bounded amplitude but ever-increasing cycle
time; asymptotically, the system cycles from being composed almost wholly of population 1, to almost wholly 2,
 to almost wholly 3, back to almost wholly 1, etc. The rigorous proof was attributed to
Schuster, Sigmund and Wolf \cite{SSW1979}.
This interesting cyclical fluctuation phenomenon was later known as ``May-Leonard phenomenon".
In \cite{CHW1998}, Chi, Hsu and Wu
obtained a complete classification for the global dynamics of the asymmetric May-Leonard competition
model, and exhibited the May-Leonard phenomenon under certain conditions.
For discrete-time cases, Roeger and Allen \cite{RA2004} studied
two discrete-time models applicable to three competing plant species
and proved that there also exist similar dynamics to system (\ref{ML}). Recently, Jiang et al. investigated
three-dimensional Kolmogorov competitive systems admitting a carrying simplex, and derived
that these systems have heteroclinic cycles
 (see \cite{JLD2014,JL2016,JL2017,JN22017}), which
implies that May-Leonard phenomenon happens and greatly enriches existing results of heteroclinic cycles
for competitive systems.
Furthermore, Jiang, Niu and Wang \cite{JNW2016} analyzed the effects of heteroclinic cycles
and the interplay of heteroclinic attractors or repellers on the boundary of the carrying simplices for
three-dimensional competitive maps.
For more results of heteroclinic cycles, we refer to
\cite{FS1991,KB1995,KB2004,OWY2008,HB2013,GJNY2018,GJNY2020} and references therein.

\par On the other hand, alternation of seasons is a very common phenomenon in nature.
The alternations may be fast (e.g., daily light change, Litchman and Klausmeier \cite{LK2001}),
intermediate (e.g., tidal cycles and annual seasons, DeAngelis et al.\cite{D2009}), or slow (e.g., EI Nino events).
Due to the seasonal alternation, populations experience a periodic
external environment such as temperature, rainfall, humidity and wind. One impressive example occurs
in phytoplankton and zooplankton of the temperate
lakes, where the species
grow during the warmer months and die off or form resting stages
in the winter. Such phenomenon is called seasonal succession in Sommer et al.\cite{SUZLD1986}.
In \cite{HT1995}, Hu and Tessier's study on competitive interactions of two Daphnia species in Gull
Lake suggested that the strength of interspecific and intraspecific competition varied with season.
This seasonal shift in the nature of competitive interactions provides an explanation for
seasonal succession in this Daphnia assemblage.
\par In fact, seasonal forcing is a major cause of nonequilibrium dynamics.
It has been a fascinating subject for ecologists and mathematicians to explore the modeling
of nonequilibrium dynamics by means of seasonal succession. As early as 1974,
Koch \cite{K1974} considered a two species competition model in a seasonal environment, where
good and bad seasons alternate. In the good season from spring to autumn, the species
follow Lotka-Volterra competition interaction. While,
the species reduce by some constant factors due to adverse conditions in the bad season from autumn to spring through winter.
This research stated that although a chance fluctuation may affect the relative extent of growth of the two species in
a particular year leading to a change in the input ratio of populations the following spring, under fairly
general circumstances, this can affect the growth cycle of the two species in just the right way to return
the system towards its original cycle so as to permit stable.
Litchman and Klausmeier \cite{LK2001} analyzed a competition model of two species for a
single nutrient under fluctuating light with seasonal succession in the chemostat. Later,
Klausmeier \cite{K2010} described a novel approach to modeling seasonally forced food webs (called successional state dynamics-SSD). The approach treats succession as a series of state transitions
driven by both the internal dynamics of species interactions and external forcing.
It is applicable to communities where species dynamics are fast relative to the external forcing, such
as plankton and other microbes, diseases, and some insect communities.
By applying this approach
to the classical Rosenzweig-MacArthur predator-prey model, he found numerically that the forced dynamics are more complicated, showing multiannual cycles
or apparent chaos.
In Steiner et al.\cite{SSH2009},
the SSD modeling approach was then employed to generate analytical predictions of the effects of altered seasonality on species
persistence and the timing of community state transitions, which highlighted the utility of the SSD modeling
approach as a framework for predicting the effects of altered seasonality on
the structure and dynamics of multitrophic communities. In theory,
Hsu and Zhao \cite{SX2012} first studied the global dynamics of a two-species Lotka-Volterra two-species competition model with
seasonal succession and obtained a complete classification for the global dynamics via the theory of monotone dynamical systems.
 Xiao \cite{X2016} discussed the effect of the seasonal harvesting on the survival of the
population for a class of population model with
seasonal constant-yield harvesting. In recent years, SSD modeling approach has attracted considerable attention in theoretical ecology,
see \cite{YX2013,TXZ2017,BBR2017,HHS2018} and references therein. Although the seasonal succession model looks
like a periodic system, its coefficients are discontinuous and periodic in time $t$. Many known
results on periodic systems with continuous periodic coefficients are not applicable to these models
(e.g. Cushing \cite{C1980}). Besides, it is still challenging to understand the mechanisms of seasonal succession and their
impacts on the dynamics of communities and ecosystems.
\par Motivated by the SSD modeling approach in Klausmeier \cite{K2010},
 we propose a three-dimensional May-Leonard competition model with seasonal succession as follows:
\begin{equation}
\begin{cases}
\label{3ML}
\frac{dx_i}{dt}=-\lambda_i x_i,\qquad \quad \qquad \ m \omega \leq t\leq m \omega+(1-\varphi)\omega, \quad i=1,2,3, \\
\frac{dx_1}{dt}=x_1(1-x_1-\alpha x_2-\beta x_3), \ m \omega+(1-\varphi)\omega \leq t\leq (m+1)\omega,\\
\frac{dx_2}{dt}=x_2(1-\beta x_1-x_2-\alpha x_3), \ m \omega+(1-\varphi)\omega \leq t\leq (m+1)\omega,\\
 \frac{dx_3}{dt}=x_3(1-\alpha x_1-\beta x_2- x_3),  \ \ m \omega+(1-\varphi)\omega \leq t\leq (m+1)\omega,\\
\big(x_1(0),x_2(0),x_3(0)\big)=x^0 \in \mathbb{R}^3_+, \quad \qquad  \ m=0,1,2\dots\\
\end{cases}
\end{equation}
where $m\in \mathbb{Z_+}$, $\varphi \in [0,1]$, $\omega$,
 $\alpha$, $\beta$, $\lambda_i \ (i=1,2,3)$ are all positive constants.
 \par Clearly, if $\varphi=0$, then the system (\ref{3ML}) become the
 following decoupling form
 \begin{equation}
 \label{LLL}
\begin{cases}
\frac{dx_i}{dt}=-\lambda_i x_i, \qquad \quad i=1,2,3,\\
  \big(x_1(0),x_2(0),x_3(0)\big)=x^0 \in \mathbb{R}^3_+.\\
\end{cases}
\end{equation}
While, if $\varphi=1$, system (\ref{3ML}) turns out to be the symmetric May-Leonard competition model (\ref{ML}).
\par From system (\ref{3ML}), one can see that it is a time-periodic system in a seasonal environment.
 Overall period is $\omega$, and
$\varphi$ represents the switching proportion of a period between two subsystems
(\ref{LLL}) and (\ref{ML}). Biologically, $\varphi$ is used to describe
the proportion of the period in the good
season in which the species follow system (\ref{ML}), while $(1-\varphi)$
stands for the proportion of the period in the bad
season in which the species die exponentially according to system (\ref{LLL}).
\par It is not difficult to see that
system (\ref{3ML}) admits a unique nonnegative
global solution $x(t,x^0)$ on $[0,+\infty)$ for any $x^0 \in \mathbb{R}^3_+$. Since system (\ref{3ML}) is
$\omega$-periodic, we only consider the Poincar$\acute{\rm e}$ map $S$ on $\mathbb{R}^3_+$,
that is, $S(x^0)=x(\omega,x^0)$ for any $ x^0 \in \mathbb{R}_+^3$. Due to system (\ref{LLL}),
let us first define a
linear map $L$ by
$$L(x_1,x_2,x_3)=(e^{-\lambda_1(1-\varphi)\omega}x_1,e^{-\lambda_2(1-\varphi)\omega}x_2,e^{-\lambda_3(1-\varphi)\omega}x_3),\
\forall \ (x_1,x_2,x_3)\in \mathbb{R}^3_+.$$
We also let $\{Q_t\}_{t\geq
0}$ represent the solution flow associated with the symmetric May-Leonard
competition system (\ref{ML}). Then, we have
$$S(x^0)=Q_{\varphi\omega}(Lx^0), \qquad \forall x^0\in \mathbb{R}^3_+, \qquad {\rm{i.e.,}} \qquad S=Q_{\varphi\omega}\circ L.$$
Hence, it suffices to investigate the dynamics of the
discrete-time system $\{S^n\}_{n\geq 0}$. However, compared to the concrete discrete-time competitive
maps discussed in \cite{OWY2008,JL2016,JL2017,JN22017,GJNY2018,GJNY2020}, there is no explicit expression of the Poincar$\acute{\rm e}$ map
$S$ for system (\ref{3ML}). This makes the research much more difficult and complicated
on the dynamics of system (\ref{3ML}).

\par The main aim of this paper is to establish the stability criterion of heteroclinic cycles
for system (\ref{3ML}), as well as its application on a given system (\ref{3ML2}).
More precisely, it is proved that system (\ref{3ML}) admits a heteroclinic cycle $\Gamma$
connecting the three axial fixed points $R_i (i=1,2,3)$ via the theory of the carrying simplex (see Theorem \ref{thm31}).
Based on this, sufficient conditions for stability of the heteroclinic cycle
are obtained (see Theorem \ref{thm32}). In particular, when $\lambda_1=\lambda_2=\lambda_3$ and $\alpha+\beta \neq 2$,
we show that system (\ref{3ML}) inherits stability properties of the heteroclinic cycle of classical May-Leonard
system (\ref{ML}). Meanwhile, we also present
the explicit expression of the carrying simplex in the case that $\alpha+\beta= 2$
(see Corollary \ref{thm33}). For a given system (\ref{3ML2}), the critical value of the existence
of the heteroclinic cycle is determined (see Theorem \ref{thm41}).
By applying the stability criterion for system (\ref{3ML}) to the given system (\ref{3ML2}), we obtain
the stability branches of the heteroclinic cycle for such a system (see Theorem \ref{thm42}).
Moreover, the branch graph of stability is demonstrated in Figure 3.
By numerical simulation, we find that there are rich dynamics for the system
as $\varphi$ varies. To be more explicit, when $\varphi$ varies in $[0,1]$, system (\ref{3ML2})
enjoys different dynamic features: the orbit emanating from initial value
$x^0=(0.5,0.6,0.4)$ for
the Poincar$\acute{\rm e}$ map $\tilde{S}$
tends to the heteroclinic cycle $\tilde{\Gamma}$(see Figure 4)
$\rightarrow$ invariant closed curves (see Figure 6 and Figure 7)
$\rightarrow$ the positive fixed point (see Figure 8-12)
$\rightarrow$ the interior fixed point of coordinate plane $\{x_2=0\}$ (see Figure 13)
$\rightarrow$ the interior fixed point of $x_1$-axis (see Figure 14)
$\rightarrow$ the interior fixed point of the coordinate plane $\{x_3=0\}$ (see Figure 15)
$\rightarrow$ the interior fixed point of $x_2$-axis (see Figure 16)
$\rightarrow$ the trivial fixed point $O$ (see Figure 17).
From these numerical examples, one can see that
the dynamic features and patterns generated by system (\ref{3ML2}) are much richer
, which are different from concrete
discrete-time competitive maps analyzed in \cite{GJNY2018,GJNY2020,JL2016,JL2017}.
\par The paper is organized as follows. In section 2, we introduce some notations and relevant definitions. Section 3 is devoted to analyze
the existence and stability of heteroclinic cycles of system (\ref{3ML}).
We investigate the stability branches of the heteroclinic cycle for a given system (\ref{3ML2})
and present the branch graph of stability in Section 4. Furthermore,
we provide some numerical simulations to exhibit rich dynamics for system (\ref{3ML2}).
This paper ends with a discussion in Section 5.

%------------------------------------
\section{Notations and Definitions}\label{sec2}
\par Throughout this paper, let $Z_+$ denote the set of nonnegative integers.
We use $\mathbb{R}^n_+$ to define the nonnegative cone
 $\mathbb{R}^n_+:=\{x\in \mathbb{R}^n:x_i\geq 0, \ i=1,\cdots,n\}$.
Let $X\subseteq \mathbb{R}^n_+$ and $S:X \rightarrow X$
be a continuous map.
The orbit of a state $x$ for $S$ is $\gamma(x)=\{S^n(x)$, $n\in
\mathbb{Z}_+ \}$. A fixed point $x$ of $S$ is a point $x\in X$ such that
$S(x)=x$. A point $y\in X$ is called a $k$-periodic point of $S$ if there
exists some positive integer $k>1$, such that $S^k(y)=y$ and
$S^m(y)\neq y$ for every positive integer $m<k$. The $k$-periodic
orbit of the $k$-periodic point $y$,
$\gamma(y)=\{y$, $S(y)$, $S^2(y)$, $...$, $S^{k-1}(y)\}$, is often called a
periodic orbit for short.
The $\omega$-{\it limit set} of $x$ is defined by $\omega(x)
=\{y\in \mathbb{R}_+^n: S^{n_k}x\rightarrow y \ (k\rightarrow\infty)\ {\rm for \ some \ sequence} \ n_k\rightarrow
+\infty \ {\rm in} \ \mathbb{Z}_+\}$. If the orbit of a state $x$ for $S^{-1}$ exists, then
the $\alpha$-{\it limit set} of $x$ is denoted by $\alpha(x)
=\{y\in \mathbb{R}_+^n: S^{-n_k}x\rightarrow y \ (k\rightarrow\infty)\ {\rm for \ some \ sequence} \ n_k\rightarrow
+\infty \ {\rm in} \ \mathbb{Z}_+\}$.
\par  A set $V \subseteq X$ is called positively
invariant under $S$ if $SV \subseteq V$, and invariant if $SV=V$.
We call that a nonempty compact positively invariant subset $V$ repels for $S$
if there exists a neighborhood $U$ of $V$ such that given any $x \in X \setminus V$
there is a $k_0(x)$ such that $S^kx \in U^c$ (the complement of $U$) for $k\geq k_0(x)$, and
$V$ is said to be an attractor for $S$,
if there exists a neighborhood $U$ of $V$, such that $\lim\limits_{k \rightarrow +\infty }d(S^kx,V)=0$
uniformly in $x\in U$, where $d(\cdot,\cdot)$ is the metric for $X$; $V$ is referred to as
a global attractor for $S$, if for any bounded set $U\subseteq X$, $\lim\limits_{k \rightarrow +\infty}
d(S^kx,V)=0$ uniformly in $x\in U$.
 Note that if the orbit $\gamma(x)$ of $x$ has
compact closure, $\omega(x)$ is nonempty, compact and invariant.
 If $S$ is a differentiable map, we write $DS(x)$ as the Jacobian matrix of
$S$ at the point $x$ and the spectral radius of $DS(x)$ is denoted by $r(DS(x))$.
 For an $n\times n$ matrix $A$, we write $A\geq
0$ iff $A$ is a nonnegative matrix (i.e., all the entries are
nonnegative) and $A\gg 0$ iff $A$ is a positive matrix (i.e., all the
entries are positive).
\par A map $S$ is said to be \textit {competitive} in $\mathbb{R}_+^n$, if
$x,y\in \mathbb{R}_+^n$ and $S(x)<S(y)$, then $x<y$. Furthermore, $S$ is called
\textit{strongly competitive} in $\mathbb{R}_+^n$, if $S(x)<S(y)$ then $x\ll y$.

\par A \textit {carrying simplex} for the periodic map $S$ is a subset $\Sigma\subset \mathbb{R}_+^n\backslash
\{0\}$ with the following properties (see \cite[Theorem 11]{YJ2002}):
\\$\rm(P1)$ $\Sigma$ is compact, invariant and unordered;
\\$\rm(P2)$ $\Sigma$ is homeomorphic via radial projection to the
$(n-1)$-dim standard probability simplex $\Delta^{n-1}:=\{x\in
\mathbb{R}_+^n |$ $\Sigma_ix_i=1\}$;
\\$\rm(P3)$ $\forall\ x\in \mathbb{R}_+^n\backslash \{0\}$, there exists some $y\in
\Sigma$ such that $\lim\limits_{n\rightarrow \infty}|S^nx-S^ny|=0$.

\section{Heteroclinic cycle}
For system (\ref{3ML}), it is clear that
$O:=(0,0,0)$ is a trivial fixed point of the Poincar$\acute{\rm e}$ map $S$. Denote $h_i:=(\varphi-\lambda_i(1-\varphi))\omega, i=1,2,3.$
By Hsu and Zhao \cite[Lemma 2.1]{SX2012}, if $h_i>0$ $(i=1,2,3)$, then $S$ has three axial fixed points, written as
$R_1:=(x_1^*,0,0)$, $R_2:=(0,x_2^*,0)$ and $R_3:=(0,0,x_3^*)$. Besides,
there are three planar fixed points for
$S$ in three coordinate planes under appropriate conditions (see \cite[Theorem 2.2-2.4]{SX2012} or \cite[Theorem 3.6]{NWX2021}).
First, we provide the existence result of carrying simplex for system (\ref{3ML}).
\begin{lemma}{\rm (The existence of the carrying
simplex)}
\label{lem20}
Assume that $h_i>0, i=1,2,3$.
 Let $S:\mathbb{R}_+^3\rightarrow \mathbb{R}_+^3$ be the Poincar$\acute{e}$ map induced by
system (\ref{3ML}). Then, system (\ref{3ML}) admits a two-dimensional carrying simplex $\Sigma$ which attracts every nontrivial orbit in $\mathbb{R}_+^3$.
\end{lemma}
\begin{proof}
The statement is a result of Niu et al. \cite[Theorem 2.3]{NWX2021} for
$n=3$, $b_i=a_{ii}=1 \ ( i=1,2,3)$,
$a_{12}=a_{23}=a_{31}=\alpha$ and $a_{32}=a_{21}=a_{13}=\beta$, where $b_i, a_{ij}, i,j=1,2,3$
are defined in \cite{NWX2021}.
\end{proof}

\par By the stability analysis of equilibria, we have
\begin{lemma}
\label{lem21}
{\rm (Stability of the fixed point $O$)} Given system (\ref{3ML}), then
\begin{enumerate}[\rm(i)]
\item If $h_i<0$ $(i=1,2,3)$, then $O$ is an asymptotically stable fixed point of $S$.
\item If one of three values $h_1,h_2$ and $h_3$ is greater than $0$,
 then $O$ is an unstable fixed point of $S$. In particular,
 if $h_i>0$ $
 (i=1,2,3)$, then $O$ is a hyperbolic repeller.
\end{enumerate}
\end{lemma}

\begin{proof}
Let $f(x)=(f_1,f_2,f_3)^T$, where

$$\left\{
\begin{array}{ll}
\displaystyle f_1=x_1(1-x_1-\alpha x_2-\beta x_3),\\
\displaystyle f_2=x_2(1-\beta x_1-x_2-\alpha x_3),\\
\displaystyle f_3=x_3(1-\alpha x_1-\beta x_2-x_3).\\
\end{array}\right.
$$
Then, $Df(x)=$
{\small $$
\left (
\begin{matrix}
1-2x_1-\alpha x_2-\beta x_3 & -\alpha x_1 & -\beta x_1 \\
-\beta x_2 & 1-2x_2-\beta x_1-\alpha x_3 & -\alpha x_2\\
-\alpha x_3 & -\beta x_3 & 1-2x_3-\alpha x_1-\beta x_2
\end{matrix}
\right ).
$$}
\\For simplicity, we denote $u(t,x):=Q_t(x)$ and $V(t,x):=D_xu(t,x)=D_xQ_t(x)$, then
$S(x)=Q_{\varphi\omega}(Lx)=u(\varphi\omega,Lx)$, and hence,
\begin{align*}
DS(x)
&=D(Q_{\varphi\omega}(Lx))\cdot
D(Lx)\\
&=V(\varphi\omega,Lx)\cdot D(Lx)\\
&=V(\varphi\omega,Lx)\cdot {\rm diag} \big({e^{-\lambda_1(1-\varphi)\omega}},{e^{-\lambda_2(1-\varphi)\omega}},{e^{-\lambda_3(1-\varphi)\omega}}
 \big).
\end{align*}
Note that $V(t,x)$ satisfies
$$\frac{dV(t)}{dt}=Df(u(t,x))V(t),\quad V(0)=I.$$
Taking $O=(0,0,0)$, we have $u(t,LO)=(0,0,0)$, then $Df(u(t,LO))=
{\rm diag} (1,1,1),$ and hence, $$V(\varphi\omega,LO)=
{\rm diag} \Big(e^{\int_0^{\varphi\omega }1dt},e^{\int_0^{\varphi\omega }1dt},e^{\int_0^{\varphi\omega} 1dt}
\Big).$$ By the expression of $S$, it follows that $$DS(O)= \left (
\begin{matrix}
e^{(\varphi-\lambda_1(1-\varphi))\omega} & 0 & 0 \\
0 & e^{(\varphi-\lambda_2(1-\varphi))\omega} & 0 \\
0 & 0 & e^{(\varphi-\lambda_3(1-\varphi))\omega}
\end{matrix}
\right ). $$ Consequently,
the matrix $DS(O)$ has three positive
eigenvalues $\mu_1,\mu_2$ and $\mu_3$ given by
$$\mu_1=e^{h_1}, \qquad \mu_2=e^{h_2} \quad  {\rm and} \quad \mu_3=e^{h_3}.$$ Clearly, if $h_i<0 $ $(i=1,2,3)$, then $\mu_i<1$, and
hence $O$ is an asymptotically stable fixed point of $S$;
if $h_i>0$, then $\mu_i>1,i=1,2,3$, which implies that $O$ is a hyperbolic repeller.
We have proved the lemma.
\end{proof}

\begin{lemma}
\label{lem22}
{(Stability of the axial fixed point $R_1$)} Given system (\ref{3ML}), suppose that $h_i>0,$
$i=1,2,3$, then
\begin{enumerate}%[(i)]
\item[(i)] If
$h_2<\beta h_1$,
$h_3<\alpha h_1$, then
$R_1$ is an asymptotically stable fixed point of $S$.
\item[(ii)] If
$h_2>\beta h_1$,
$h_3>\alpha h_1$, then $R_1$ is a saddle point with two-dimensional unstable manifolds.
\item[(iii)] If
$(h_2-\beta h_1)(h_3-\alpha h_1)<0$,
then $R_1$ is a saddle point with one-dimensional unstable manifolds.
\end{enumerate}
\end{lemma}
\begin{proof}
Let $u(t,LR_1):=(u_1^*(t),0,0)$. By the proof of Lemma \ref{lem21}, it is easy to see that
$$Df(u(t,LR_1))= \left (
\begin{matrix}
1-2u_1^*(t) & -\alpha u_1^*(t) & -\beta u_1^*(t) \\
0 & 1-\beta u_1^*(t) & 0 \\
0 & 0 & 1-\alpha u_1^*(t)
\end{matrix}
\right ). $$ Then,
$$V(\varphi\omega,LR_1)= \left (
\begin{matrix}
e^{\int_0^{\varphi\omega}1-2u_1^*(t)dt} & * & *\\
0 & e^{\int_0^{\varphi\omega}1-\beta u_1^*(t)dt} & * \\
0 & 0 & e^{\int_0^{\varphi\omega}1-\alpha u_1^*(t)dt}
\end{matrix}
\right ),$$
where $*$ represents unknown algebraic expressions. Observe that $u_1^*(t)$ satisfies
$$\frac{du_1^*(t)}{dt}=u_1^*(t)(1-u_1^*(t)).$$
Integrating the above equation for $t$ from $0$ to $\varphi\omega$, we have
$$\int_0^{\varphi\omega}u_1^*(t)dt=(\varphi-\lambda_1(1-\varphi))\omega,$$
and then, $DS(R_1)=$
$$ \left (
\begin{matrix}
e^{-(\varphi-\lambda_1(1-\varphi))\omega} & * & *\\
0 & e^{[(\varphi-\lambda_2(1-\varphi))-\beta (\varphi-\lambda_1(1-\varphi))]\omega} & * \\
0 & 0 &
e^{[(\varphi-\lambda_3(1-\varphi))-\alpha(\varphi-\lambda_1(1-\varphi))]\omega}
\end{matrix}
\right ).$$
Thus,
the matrix $DS(R_1)$ has three positive
eigenvalues $\mu_1,\mu_2$ and $\mu_3$ given by
$$\mu_1=e^{-h_1}<1, \qquad \mu_2=e^{h_2-\beta h_1} \quad {\rm and} \quad \mu_3=e^{h_3-\alpha h_1},$$
which implies that the statements (i)-(iii) are valid. The proof is completed.
\end{proof}
\par By symmetric arguments, we also have
\begin{lemma}
\label{lem23}
(Stability of the axial fixed point $R_2$) Given system (\ref{3ML}), suppose that $h_i>0,
i=1,2,3$, then
\begin{enumerate}%[(i)]
\item[(i)] If
$h_1<\alpha h_2$,
$h_3<\beta h_2$, then
$R_2$ is an asymptotically stable fixed point of $S$.
\item[(ii)] If
$h_1>\alpha h_2$,
$h_3>\beta h_2$, then $R_2$ is a saddle point with two-dimensional unstable manifolds.
\item[(iii)] If
$(h_1-\alpha h_2)(h_3-\beta h_2)<0$, then
$R_2$ is a saddle point with one-dimensional unstable manifolds.
\end{enumerate}
\end{lemma}

\begin{lemma}
\label{lem24}
(Stability of the axial fixed point $R_3$) Given system (\ref{3ML}), suppose that $h_i>0,
i=1,2,3$, then
\begin{enumerate}%[(i)]
\item[(i)] If
$h_1<\beta h_3$,
$h_2<\alpha h_3$, then
$R_3$ is an asymptotically stable fixed point of $S$.
\item[(ii)] If
$h_1>\beta h_3$,
$h_2>\alpha h_3$, then $R_3$ is a saddle point with two-dimensional unstable manifolds.
\item[(iii)] If
$(h_1-\beta h_3)(h_2-\alpha h_3)<0$, then
$R_3$ is a saddle point with one-dimensional unstable manifolds.
\end{enumerate}
\end{lemma}

\par For the sake of convenience, we introduce the following notation:
\vspace{2mm}
\\\textbf{$(\tilde{H})$} \quad \qquad \qquad  $h_2<\beta h_1, h_1>\alpha h_2;\ h_3<\beta h_2, h_2>\alpha h_3;\ h_1<\beta h_3, h_3>\alpha h_1.$
\vspace{2mm}
\\Moreover, we write $\vartheta$ as
$$\vartheta:=(h_1-\beta h_3)(h_2-\beta h_1)(h_3-\beta h_2)+(h_1-\alpha h_2)(h_2-\alpha h_3)(h_3-\alpha h_1).$$
\par By the theory of the carrying simplex, we have
\begin{theorem}
\label{thm31}
{\rm (The existence of the heteroclinic cycle)}
Suppose that the condition
$(\tilde{H})$ holds and $h_i>0, i=1,2,3$, then system (\ref{3ML}) admits a heteroclinic cycle
$\Gamma$ ($=\partial \Sigma$) connecting the three axial fixed points $R_1, R_2$ and $R_3$
, where $\partial \Sigma$ stands for the boundary of the carrying simplex $\Sigma$
(see Figure 1).

\end{theorem}
\begin{proof}
By Lemma \ref{lem20}, system (\ref{3ML}) admits a two-dimensional carrying simplex $\Sigma$.
Furthermore, the intersection of $\Sigma$ with three coordinate planes is composed of three line
segments connecting any two of $R_i,i=1,2,3$. These three line segments are written as
$\Gamma_{12},\Gamma_{23}$ and $\Gamma_{13}$, respectively. By the invariance of $\Sigma$, one can obtain that
$\Gamma_{12},\Gamma_{23}$ and $\Gamma_{13}$ are invariant under the map $S$.
In view of $h_2<\beta h_1$ and $ h_1>\alpha h_2$, it follows from Niu et al. \cite[Theorem 3.6(ii)]{NWX2021} that
there are no interior fixed points on $\Gamma_{12}$.
Note that the restriction of $S$ to the segment $\Gamma_{12}$ must be a monotone one-dimensional map,
all orbits on $\Gamma_{12}$ have one fixed point as the $\alpha$-limit set and the other as the $\omega$-limit
set. Meanwhile,
Niu et al. \cite[Theorem 3.6(ii)]{NWX2021} ensures that all orbits on $\Gamma_{12}$ take $R_2$ as the $\alpha$-limit set
and $R_1$ as the $\omega$-limit set under the condition $(\tilde{H})$. By the cyclic symmetry, we also
obtain that $\Gamma_{23}$ and $\Gamma_{13}$ have similar properties
as $\Gamma_{12}$ under the condition $(\tilde{H})$. Define
$\Gamma=\Gamma_{12}\cup \Gamma_{13} \cup \Gamma_{23}$, then $\Gamma= \partial \Sigma$. Furthermore, it is obvious that
$\Gamma$ is a heteroclinic cycle connecting the three axial fixed points
$R_1,R_2$ and $R_3$, which forms the boundary  $\partial \Sigma$
of $\Sigma$ (see Figure 1). We have thus proved the theorem.
\end{proof}

\begin{figure}[htbp]
\centering
\begin{minipage}[t]{0.48\textwidth}
\centering
\includegraphics[width=6cm]{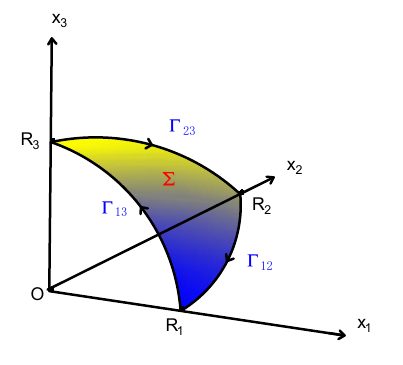}
\caption{\footnotesize  $\Gamma$ (Clockwise direction)}
\end{minipage}
\begin{minipage}[t]{0.48\textwidth}
\centering
\includegraphics[width=6cm]{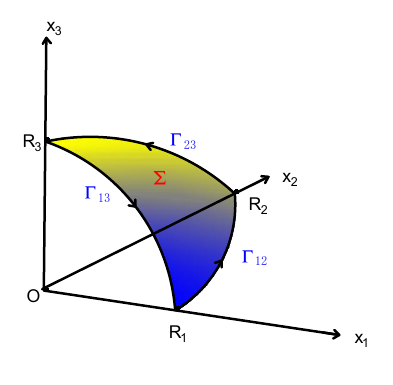}
\caption{\footnotesize  $\Gamma$ (Counter-clockwise direction)}
\end{minipage}
\end{figure}

\begin{remark}
\label{rmk31}
By modifying the proof of Theorem \ref{thm31}, it can easily be shown that
if the condition $(\tilde{H})$ is revised as
$$h_2>\beta h_1, h_1<\alpha h_2;\ h_3>\beta h_2, h_2<\alpha h_3;\ h_1>\beta h_3, h_3<\alpha h_1,$$
then system (\ref{3ML}) also admits a heteroclinic cycle connecting the three axial fixed points
$R_1,R_2$ and $R_3$, yet the arrows are reserved (see Figure 2).
\end{remark}

\begin{theorem}
{\rm (The stability of the heteroclinic cycle)}
\label{thm32}
Given system (\ref{3ML}), suppose that the condition $(\tilde{H})$ holds and $h_i>0,i=1,2,3$, then
\begin{enumerate}
\item [(i)]If $\vartheta >0$, then
the heteroclinic cycle $\Gamma$ repels.
\item [(ii)]If $\vartheta <0$, then the heteroclinic cycle
$\Gamma$ attracts.
\end{enumerate}
\end{theorem}
\begin{proof}
By the proof of
Lemma \ref{lem22} and Lemma \ref{lem23}-\ref{lem24}, one can see that the eigenvalues of the
matrices $DS(R_1)$ and $DS(R_3)$ along
the $x_2$-axis direction are $e^{h_2-\beta h_1}$ and $e^{h_2-\alpha h_3}$, respectively;
 the eigenvalues of the
matrices $DS(R_1)$ and $DS(R_2)$ along
the $x_3$-axis direction are $e^{h_3-\alpha h_1}$ and $e^{h_3-\beta h_2}$, respectively;
 the eigenvalues of the
matrices $DS(R_2)$ and $DS(R_3)$ along
the $x_1$-axis direction are $e^{h_1-\alpha h_2}$ and $e^{h_1-\beta h_3}$, respectively.
A simple calculation yields
$$\ln e^{h_1-\beta h_3}\cdot \ln e^{h_2-\beta h_1}\cdot \ln e^{h_3-\beta h_2}+ \ln e^{h_1-\alpha h_2}\cdot
\ln e^{h_2-\alpha h_3} \cdot \ln e^{h_3-\alpha h_1}  $$
$$=(h_1-\beta h_3)(h_2-\beta h_1)(h_3-\beta h_2)+(h_1-\alpha h_2)(h_2-\alpha h_3)(h_3-\alpha h_1)=\vartheta.$$
Note that system (\ref{3ML}) is a time-periodic Kolmogorov systems and
the the Poincar$\rm \acute{e}$ map $S$ is $C^1$-diffeomorphism onto its image
(see the proof of Theorem 2.3 in Niu et al. \cite{NWX2021}), we can rewrite $S$ as:
$$S(x_1,x_2,x_3)=(x_1G_1(x),x_2G_2(x),x_3G_3(x)), \quad   x\in \mathbb{R}_+^3,$$
where $G_i(x)=$ $\left\{
\begin{array}{ll}
\displaystyle \frac{S_i(x)}{x_i},  \quad x_i \neq 0, \\
\displaystyle \frac{\partial S_i}{\partial x_i}(x),\quad x_i=0.\\
\end{array}\right.$
Then $S$ is considered as a Kolmogorov competitive map. According to Theorem 3 in Jiang, Niu and Wang \cite{JNW2016},
we can obtain that if
$\vartheta>0$, then the heteroclinic cycle $\Gamma$ repels; if
$\vartheta<0$, then the heteroclinic cycle $\Gamma$ attracts. Thus, we have completed the proof.
\end{proof}

\par In particular, when $\lambda_1=\lambda_2=\lambda_3$, we have
\begin{corollary}
\label{thm33}
Given system (\ref{3ML}), suppose that the condition $(\tilde{H})$ holds,
$h_i>0 \ (i=1,2,3)$, $\lambda_1=\lambda_2=\lambda_3=\lambda$ and $0<\alpha<1<\beta$
, then
\begin{enumerate}
\item [(i)] If $\alpha+\beta<2$, then the heteroclinic cycle $\Gamma$ repels.
\item [(ii)] If $\alpha+\beta>2$, then the heteroclinic cycle $\Gamma$ attracts.
\item [(iii)] If $\alpha+\beta=2$, then the carrying simplex of system (\ref{3ML}) is
$$\Sigma=\{(x_1,x_2,x_3)\in \mathbb{R}_+^3: x_1+x_2+x_3=\frac{1-e^{p-q}}{1-e^{-q}}\},$$
where $p=\lambda(1-\varphi)\omega, \ q=\varphi\omega$. Furthermore,
$P:=\frac{1-e^{p-q}}{3(1-e^{-q})}(1,1,1)$ is a positive fixed point of
$S$.
\end{enumerate}
\end{corollary}
\begin{proof}
Since $\lambda_1=\lambda_2=\lambda_3$, it follows that $h_1=h_2=h_3$, and then,
\begin{align*}
\vartheta &=h_1^3\big((1-\beta)^3+(1-\alpha)^3\big)\\
&=h_1^3\Big(\big(2-(\alpha+\beta)\big)\big( (1-\beta)^2 +(\beta-1)(1-\alpha)+(1-\alpha)^2  \big)\Big).
\end{align*}
Noticing that $0<\alpha<1<\beta,$ one can obtain that $\big( (1-\beta)^2 +(\beta-1)(1-\alpha)+(1-\alpha)^2  \big)>0.$
By $h_i>0\ (i=1,2,3)$, it is clear that if $\alpha+\beta<2$, then
$\vartheta >0$; if $\alpha+\beta>2$, then
$\vartheta <0$. This implies that the statements (i) and (ii) are true due to Theorem \ref{thm32}.
\par Finally, it remains to show that the statement (iii) holds. For this purpose, we firstly consider the one-dimensional case of system (\ref{3ML}),
that is,
\begin{equation}
\label{3ML-1}
\begin{cases}
\frac{dx}{dt}=-\lambda x,\qquad \quad \; m \omega \leq t\leq m \omega+(1-\varphi) \omega,\\
\frac{dx}{dt}=x(1-x),\quad m \omega+(1-\varphi) \omega \leq t\leq (m+1)\omega,\\
x(0)=x_0 \in \mathbb{R}_+.
\end{cases}
\end{equation}
Solving the equation $\frac{dx}{dt}=x(1-x)$ with the initial value $x^0 \in \mathbb{R}_+$ yields
\begin{equation}
\label{3ML-2}
x(t,x^0)=\frac{x^0}{x^0+(1-x^0)e^{-t}}, \qquad \forall \ t\geq 0,\ x^0 \in \mathbb{R}_+.
\end{equation}
Define the Poincar$\rm \acute{e}$ map of system (\ref{3ML-1}),
$$\bar{S}(x_0)=x(\omega,x_0), \ \forall \ x_0 \in \mathbb{R}_+.$$
By Hsu and Zhao\cite[Lemma 2.1]{SX2012}, $\bar{S}$ has a unique positive fixed point $x^*$
as $\varphi-\lambda(1-\varphi)>0$.
Based on the expression (\ref{3ML-2}) and a straightforward computation, it then follows that
\begin{equation}
\label{3ML-3}
x^*=\frac{1-e^{\lambda(1-\varphi)\omega-\varphi\omega}}{1-e^{-\varphi\omega}}=\frac{1-e^{p-q}}{1-e^{-q}}.
\end{equation}
\par Since $h_i>0\ (i=1,2,3)$, the Poincar$\rm \acute{e}$ map $S$ generated by system (\ref{3ML})
has three axial fixed points $R_1, R_2$ and $R_3$.
Note that $\lambda_1=\lambda_2=\lambda_3=\lambda$ and (\ref{3ML-3}), it follows that
$$R_1=(\frac{1-e^{p-q}}{1-e^{-q}},0,0),\ \ R_2=(0,\frac{1-e^{p-q}}{1-e^{-q}},0),\ \ R_3=(0,0,\frac{1-e^{p-q}}{1-e^{-q}}).$$
Then, the intersection of $\mathbb{R}_+^3$ with
the planar including three axial fixed points $R_1$,$R_2$ and $R_3$ is
$$\pi:=\{(x_1,x_2,x_3)\in \mathbb{R}_+^3: x_1+x_2+x_3=\frac{1-e^{p-q}}{1-e^{-q}}\}.$$
For any $\bar{x}=(a,b,c)$ $\in \pi$, we have $a+b+c=\frac{1-e^{p-q}}{1-e^{-q}},$
and
$$L\bar{x}=(e^{-\lambda(1-\varphi)\omega}a,e^{-\lambda(1-\varphi)\omega}b,e^{-\lambda(1-\varphi)\omega}c)
=(e^{-p}a,e^{-p}b,e^{-p}c).$$
Let $Q_t(x):=\big(x_1(t,x),x_2(t,x),x_3(t,x)\big)$ represent the solution flow associated with system
(\ref{ML}), then
$$Q_t(L\bar{x})=\big(x_1(t,L\bar{x}),x_2(t,L\bar{x}),x_3(t,L\bar{x})\big),$$
and hence, $$S(\bar{x})=Q_{\varphi\omega}(L\bar{x})=\big(x_1(\varphi\omega,L\bar{x}),x_2(\varphi\omega,L\bar{x}),x_3(\varphi\omega,L\bar{x})\big).$$
\par Next, we will show that $S(\bar{x})\in \pi$. For system (\ref{ML}), we first denote
$$ Y(t)=x_1(t)+x_2(t)+x_3(t),$$ and then, $$\frac{dY(t)}{dt}=Y(t)(1-Y(t)).$$
Solving the above equation with the initial value $Y(0)\in \mathbb{R}_+^3$ gives
$$Y(t)=\frac{Y(0)}{Y(0)+(1-Y(0))e^{-t}}, \qquad \quad \forall \ t\geq 0,\ Y(0)\in \mathbb{R}_+^3.$$
After a easy calculation, we have $$Y(0)=x_1(0,L\bar{x})+x_2(0,L\bar{x})+x_3(0,L\bar{x})=e^{-p}a+e^{-p}b+e^{-p}c=e^{-p}(a+b+c).$$
Taking $t=\varphi\omega$,
it follows that
$$x_1(\varphi\omega, L\bar{x})+x_2(\varphi\omega, L\bar{x})+x_3(\varphi\omega, L\bar{x})
=Y(\varphi\omega)=\frac{Y(0)}{Y(0)+(1-Y(0))e^{-\varphi\omega}}.$$
Then,
$$Y(\varphi\omega)=\frac{e^{-p}(a+b+c)}{e^{-p}(e^{-p}(a+b+c))
+e^{-q}-e^{-p-q}(e^{-p}(a+b+c))}.$$ Substituting
$a+b+c=\frac{1-e^{p-q}}{1-e^{-q}}$ into the above equation yields
$$x_1(\varphi\omega, L\bar{x})+x_2(\varphi\omega, L\bar{x})+x_3(\varphi\omega, L\bar{x})=Y(\varphi\omega)=\frac{1-e^{p-q}}{1-e^{-q}},$$
which implies that $S(\bar{x})\in \pi$.
By the arbitrariness of $\bar{x}$, one can see that $\pi$ is invariant under the map
$S$. Using the invariance of the carrying simplex $\Sigma$, we further obtain
that $\pi$ is the carrying simplex $\Sigma$ of system (\ref{3ML}), that is,
$$\Sigma=\pi=\{(x_1,x_2,x_3)\in \mathbb{R}_+^3: x_1+x_2+x_3=\frac{1-e^{p-q}}{1-e^{-q}}\}.$$
\par In particular, when $\varphi=1$, the carrying simplex $\Sigma'$ of the classical May-Leonard competition model
(\ref{ML}) is
$$\Sigma'=\{(x_1,x_2,x_3)\in \mathbb{R}_+^3: x_1+x_2+x_3=1\}.$$
Under the condition that $0<\alpha<1<\beta$ and $\alpha+\beta=2$,
system (\ref{ML}) admits a unique positive equilibria $P':=(\frac{1}{3},\frac{1}{3},\frac{1}{3})$. Then
the ray equation connecting $O$ and $P'$ is $$l=\{(\frac{1}{3}t,\frac{1}{3}t,\frac{1}{3}t)| t\geq 0 \}.$$ Next,
we prove that
$l$ is invariant under system (\ref{ML}). Clearly, $Q_t(O)=O$, where $\{Q_t\}_{t\geq 0}$ stands for the solution flow
associated with system (\ref{ML}). For any
$x_0 \in l\setminus \{O\}$, there exists $\xi>0$ such that $x_0=(\frac{1}{3}\xi,\frac{1}{3}\xi,\frac{1}{3}\xi).$ Let $\phi(t)$
be the unique solution of the following system
\begin{equation}
\nonumber
\begin{cases}
\frac{d\phi(t)}{dt}=\phi(t)(1-\phi(t)),\\
 \phi(0)=\xi.
\end{cases}
\end{equation}
It can easily be verified that $(\frac{1}{3}\phi(t),\frac{1}{3}\phi(t),\frac{1}{3}\phi(t))$ is
the unique solution of system (\ref{ML}) with initial value
$x_0=(\frac{1}{3}\xi,\frac{1}{3}\xi,\frac{1}{3}\xi)$. Based on the fact that
$(\frac{1}{3}\phi(t),\frac{1}{3}\phi(t),\frac{1}{3}\phi(t))\in l\setminus \{O\}$,
it follows that $l$ is invariant under system (\ref{ML}).
\par On the other hand, there admits a unique intersection point $P=\frac{1-e^{p-q}}{3(1-e^{-q})}(1,1,1)$
between $l$ and the carrying simplex $\Sigma$ of system (\ref{3ML}), that is,
$l \cap \Sigma =\{P\}.$ Recall that the expression
$L(x_1,x_2,x_3)=e^{-p}(x_1,x_2,x_3),$ we can obtain that $LP \in l$. Again, $l$ is invariant
under the solution flow $\{ Q_t \}_{t\geq 0}$ of system
(\ref{ML}), we have $Q_t(LP)\in l$. Taking $t=\varphi\omega,$ it follows that
$S(P)=Q_{\varphi \omega}(LP) \in l.$ By the invariance of the carrying simplex, we also get
$S(P)\in \Sigma$, and then, $S(P)\in l \cap \Sigma$. Consequently, $S(P)=P$, which implies that
$P$ is a positive fixed point of $S$. We have completed the proof.
\end{proof}
\begin{remark}
From Corollary \ref{thm33}, one can see that system (\ref{3ML}) inherits stability properties of the heteroclinic cycle of the
classical May-Leonard competition model (\ref{ML}) in the case that
$\lambda_1=\lambda_2=\lambda_3$ and $\alpha+\beta \neq 2$. This means that
the stability of the heteroclinic cycle has not been changed by
the introduction of seasonal succession.
\end{remark}

\section{Bifurcation and Numerical simulation}
There are five parameters $\varphi, \omega, \lambda_i \ (i=1,2,3) $ related to the seasonal succession.
By Theorem \ref{thm31}-\ref{thm32},
it is not difficult to see that the period $\omega$ is independent of the existence and stability of the heteroclinic cycle.
 Hsu and Zhao\cite[Lemma 2.1]{SX2012} states that if $h_i\leq 0, i=1,2,3$ (i.e., $\lambda_i \geq
\frac{\varphi}{1-\varphi}$), then there does not exist
axial fixed points for system (\ref{3ML}), which implies that there is no heteroclinic cycles connecting three axial fixed points.
Here, we select $\varphi$ as the switching strategy parameter for
a given example to analyze the stability of the heteroclinic cycle.
By numerical simulation, rich dynamics (including
nontrivial periodic solutions, invariant closed curves and heteroclinic cycles) for such a system are exhibited.
\par For system (\ref{ML}), when taking $\alpha=0.4, \beta=1.61$, it follows from May and Leonard\cite{ML1975}
(or Chi, Hsu and Wu \cite{CHW1998}) that
the heteroclinic cycle of the system is attracting. In what follows, for the same parameter values
$\alpha=0.4$ and $\beta=1.61$, we also take
$\lambda_1=0.2, \lambda_2=0.1, \lambda_3=0.4,$
$\omega=10$ and initial values $x_1(0)=0.5, x_2(0)=0.6, x_3(0)=0.4$, and then, the concrete form is given by
\begin{equation}
\begin{cases}
\label{3ML2}
\frac{dx_1}{dt}=-0.2 x_1,\qquad \quad \qquad \qquad \qquad \ \ 10m  \leq t\leq 10m+10(1-\varphi), \\
\frac{dx_1}{dt}=x_1(1-x_1-0.4  x_2-1.61 x_3), \ \ 10m +10(1-\varphi) \leq t\leq 10(m+1),\\
\frac{dx_2}{dt}=-0.1 x_2,\qquad \quad \qquad \qquad \qquad\ \ 10m \leq t\leq 10m+10(1-\varphi), \\
\frac{dx_2}{dt}=x_2(1-1.61 x_1-x_2-0.4 x_3),\ \ 10m+10(1-\varphi) \leq t\leq 10(m+1),\\
\frac{dx_3}{dt}=-0.4 x_3,\qquad \quad \qquad  \qquad \qquad\ \ 10m \leq t\leq 10m+10(1-\varphi), \\
 \frac{dx_3}{dt}=x_3(1-0.4 x_1-1.61 x_2- x_3),  \ \ 10m +10(1-\varphi) \leq t\leq 10(m+1),\\
\big(0.5,0.6,0.4\big)=x^0 \in \mathbb{R}^3_+, \quad \qquad \qquad \qquad \ m=0,1,2\dots,\\
\end{cases}
\end{equation}
where $\varphi \in [0,1]$. Recall that $h_i>0 \ (i=1,2,3)$ implies that system (\ref{3ML2}) has three interior fixed points
of $i$-axis, written as $\tilde{R}_i \ (i=1,2,3)$. Then, we have
\begin{theorem}
\label{thm41}
Let $\tilde{S}$ be the Poincar$\acute{ e}$ map
induced by system (\ref{3ML2}),
then
\begin{enumerate}
\item [(i)] If $\varphi \in [\frac{4.44}{10.54},1]$, then system (\ref{3ML2}) admits a
heteroclinic cycle $\tilde{\Gamma}$ connecting three axial fixed points $\tilde{R}_i \ (i=1,2,3)$.
\item [(ii)] If $\varphi \in [0,$ $\frac{4.44}{10.54})$, then
system (\ref{3ML2}) does not admit a
heteroclinic cycle connecting three axial fixed points $\tilde{R}_i \ (i=1,2,3)$.
\end{enumerate}
\end{theorem}
\begin{proof}
Noticing that $h_i=(\varphi-\lambda_i(1-\varphi))\omega \ (i=1,2,3)$, an easy calculation yields
$$h_1=12\varphi-2, \ h_2=11\varphi-1 \ \ {\rm and} \ \ h_3=14\varphi-4.$$
In addition,
$$h_2-\beta h_1=2.22-8.32\varphi,\ h_1-\alpha h_2=7.6\varphi-1.6,\ h_3-\beta h_2=-3.71\varphi-2.39<0$$
and
$$h_2-\alpha h_3=5.4\varphi+0.6>0 ,\ h_1-\beta h_3=-10.54\varphi+4.44, \ h_3-\alpha h_1=9.2\varphi-3.2 .$$
It then easily follows that
\begin{enumerate}
\item [(1)] $\varphi>\frac{2}{7}$  if and only if $h_i>0 \ (i=1,2,3)$;
\item [(2)] $\varphi>\frac{4.44}{10.54}$ if and only if the condition $(\tilde{H})$ holds.
\end{enumerate}
This implies that if $\varphi>\frac{4.44}{10.54}$, then system (\ref{3ML2}) satisfies $(\tilde{H})$ and $h_i>0,i=1,2,3$.
Meanwhile, system (\ref{3ML2}) has three axial fixed points $\tilde{R}_i \ (i=1,2,3)$.
By Theorem \ref{thm31}, we can see that
if $\varphi>\frac{4.44}{10.54}$, then system (\ref{3ML2}) admits a
heteroclinic cycle $\tilde{\Gamma}$ connecting three axial fixed points $\tilde{R}_i \ (i=1,2,3)$.
\par When $\varphi=\frac{4.44}{10.54}$, one can calculate that
$$h_2-\beta h_1<0,\ h_1-\alpha h_2>0,\ h_3-\beta h_2<0,
h_2-\alpha h_3>0 ,\ h_1=\beta h_3, \ h_3-\alpha h_1>0$$
and $h_i>0 \ (i=1,2,3)$. Compared with the condition $(\tilde{H})$, we only consider the case that
$h_1=\beta h_3, \ h_3-\alpha h_1>0$. Since $h_1=\beta h_3$, it follows from Hsu and Zhao \cite[Lemma 2.5(i)]{SX2012}
that there is no interior fixed points on the coordinate plane $\{x_2=0\}$. By Lemma \ref{lem20},
we denote $\tilde{\Sigma}$ to be the carrying simplex of system (\ref{3ML2}) and
$\tilde{\Gamma}_{13}$ to be the intersection of $\tilde{\Sigma}$ with the coordinate plane
$\{x_2=0\}$. Due to $h_3-\alpha h_1>0$ and Lemma \ref{lem22} (or see Niu et al. \cite[Lemma 3.2]{NWX2021}), we can see that $\tilde{R}_1$ repels along $\tilde{\Gamma}_{13}$.
Note that the restriction of $\tilde{S}$ to the segment $\Gamma_{13}$ is a monotone one-dimensional map,
all orbits on $\Gamma_{13}$ have one fixed point as the $\alpha$-limit set and the other as the $\omega$-limit
set, and then, all orbits on $\Gamma_{13}$ take $R_1$ as the $\alpha$-limit set
and $R_3$ as the $\omega$-limit set. By the cyclic symmetry, it is not difficult to obtain that
system (\ref{3ML2}) admits a
heteroclinic cycle $\tilde{\Gamma}$ connecting three axial fixed points $\tilde{R}_i \ (i=1,2,3)$.
\par On the other hand, when $\varphi \in [0,\frac{2}{7}]$, it follows from Hsu and Zhao\cite[Lemma 2.1]{SX2012}
that there is no interior fixed points on $x_3$-axis, which entails that there does not admit a
heteroclinic cycle connecting three axial fixed points in this case.
When $\varphi \in (\frac{2}{7},\frac{3.2}{9.2}]$, we have
$$h_2-\beta h_1=2.22-8.32\varphi<0, \ h_1-\alpha h_2=7.6\varphi-1.6>0,\ h_3-\beta h_2<0,$$
and $$
h_2-\alpha h_3>0,\
h_1-\beta h_3=-10.54\varphi+4.44>0, \ h_3-\alpha h_1=9.2\varphi-3.2 \leq 0 ;$$
when $\varphi \in (\frac{3.2}{9.2},\frac{4.44}{10.54})$, we also have
$$ h_2-\beta h_1=2.22-8.32\varphi<0, \ h_1-\alpha h_2=7.6\varphi-1.6>0, \ h_3-\beta h_2<0,$$
and $$
h_2-\alpha h_3>0,\
h_1-\beta h_3=-10.54\varphi+4.44> 0, \ h_3-\alpha h_1=9.2\varphi-3.2 > 0.$$
For these two cases, by using the proof of Theorem \ref{thm31} and
Lemma \ref{lem22}-\ref{lem24}, it is not difficult to see that
the boundary of the carrying simplex $\tilde{\Sigma}$ of
system (\ref{3ML2}) can not form  a
heteroclinic cycle connecting three axial fixed points.
Consequently, when $\varphi \in [0,\frac{4.44}{10.54})$, there does not exist a
heteroclinic cycle connecting three axial fixed points  $\tilde{R}_i \ (i=1,2,3)$ for system (\ref{3ML2}).
The proof is completed.
\end{proof}
Noticing that
$$\vartheta=(7.6\varphi-1.6)(5.4\varphi+0.6)(9.2\varphi-3.2)-(8.32\varphi-2.22)(3.71\varphi+2.39)(10.54\varphi-4.44).$$
Solving the algebra equation $\vartheta=0$, there exists a unique positive root
$\varphi_0 \approx 0.6995$ in $\varphi \in [\frac{4.44}{10.54},1]$. Moreover, when $\varphi \in $ $[\frac{4.44}{10.54},1]$, we have
\begin{theorem}
\label{thm42}
Given system (\ref{3ML2}), then
\begin{enumerate}
\item [(i)] If $\varphi \in (\varphi_0,1]$, then $\tilde{\Gamma}$ attracts;
\item [(ii)] If $\varphi \in [\frac{4.44}{10.54},\varphi_0)$, then $\tilde{\Gamma}$ repels,
\end{enumerate}
where $\varphi_0$ is the positive root of the algebra equation $\vartheta=0$.
The bifurcation graph of $\varphi$-$\vartheta$ in $\varphi \in [\frac{4.44}{10.54},1]$ is
demonstrated in Figure 3.
\end{theorem}
\begin{proof}
For the algebra equation $\vartheta=0$, an easy calculation gives
\begin{enumerate}
\item [(1)] If $\varphi \in (\varphi_0,1]$, then $\vartheta<0$.
\item [(2)] If $\varphi \in [\frac{4.44}{10.54},\varphi_0)$, then $\vartheta>0$.
\end{enumerate}
By Theorem \ref{thm32}, it follows that the statements (i) and (ii) are valid,
which implies that $\varphi_0$ is the critical value
of stability of the heteroclinic cycle $\tilde{\Gamma}$. The proof is completed.
\end{proof}
Consider that $\frac{4.44}{10.54}\approx 0.4213$ and $\varphi_0 \approx 0.6995$,
it follows from Theorem \ref{thm42} that the bifurcation graph of $\varphi$-$\vartheta$ is
shown as below (see Figure 3).
\begin{figure}[htbp]
  \centering
  \includegraphics[width=7cm]{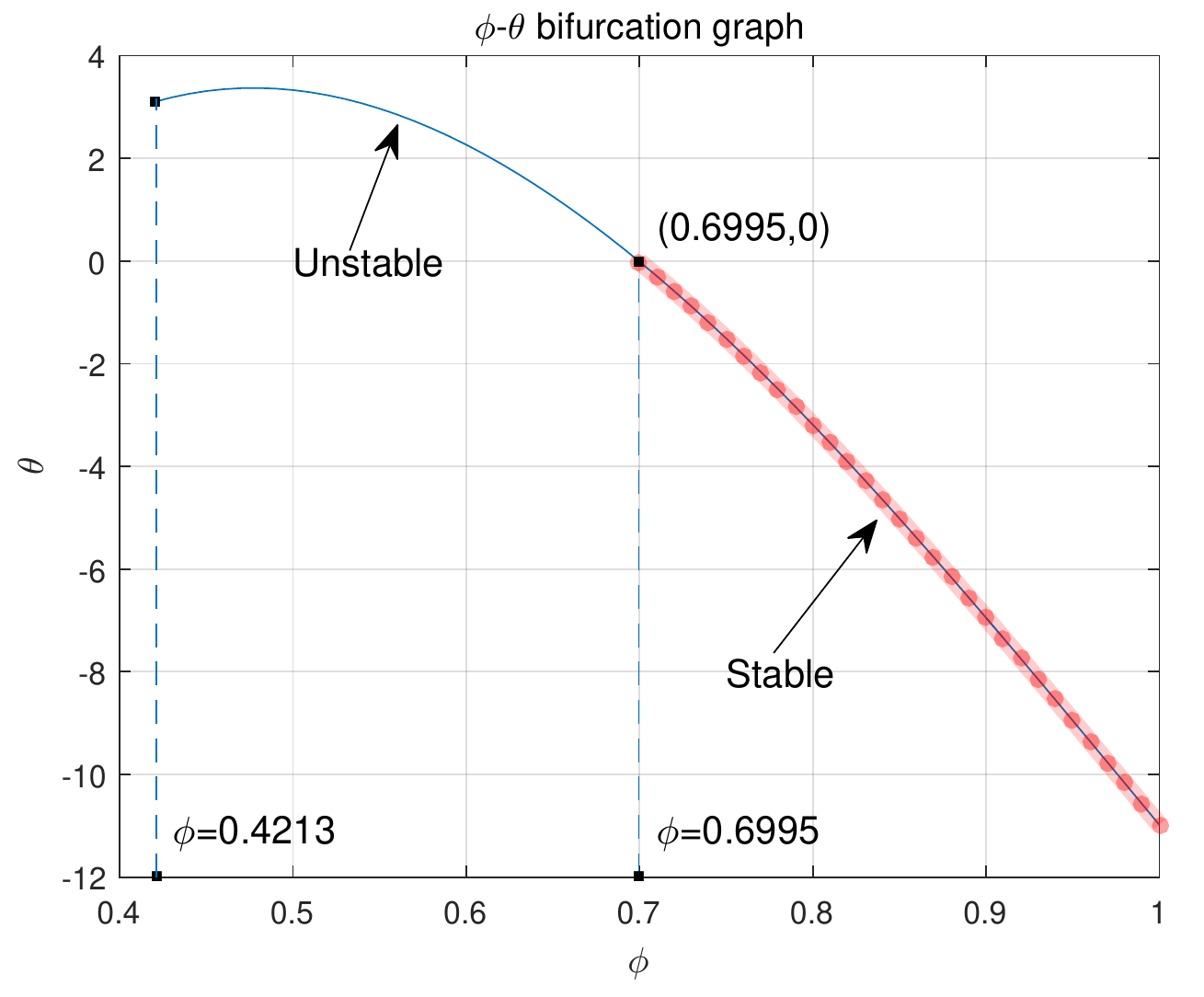}
  \caption{$\varphi$-$\vartheta$ bifurcation graph, $\varphi \in [\frac{4.44}{10.54},1]$}
\end{figure}
\par According to above analyses, we take different values of $\varphi$ to observe
the dynamics of system (\ref{3ML2}) via numerical simulation.
Firstly, taking $\varphi=0.73 $, one can see that
the orbit emanating from initial value
$x^0=(0.5,0.6,0.4)$ for
the Poincar$\acute{\rm e}$ map $\tilde{S}$
generated by system (\ref{3ML2})
tends to the heteroclinic cycle $\tilde{\Gamma}$ from Figure 4. The corresponding
solution flow is indicated in Figure 5, which states that system (\ref{3ML2}) possesses the
May-Leonard phenomenon.
\begin{figure}[htbp]
\centering
\begin{minipage}[t]{0.48\textwidth}
\centering
\includegraphics[width=6.5cm]{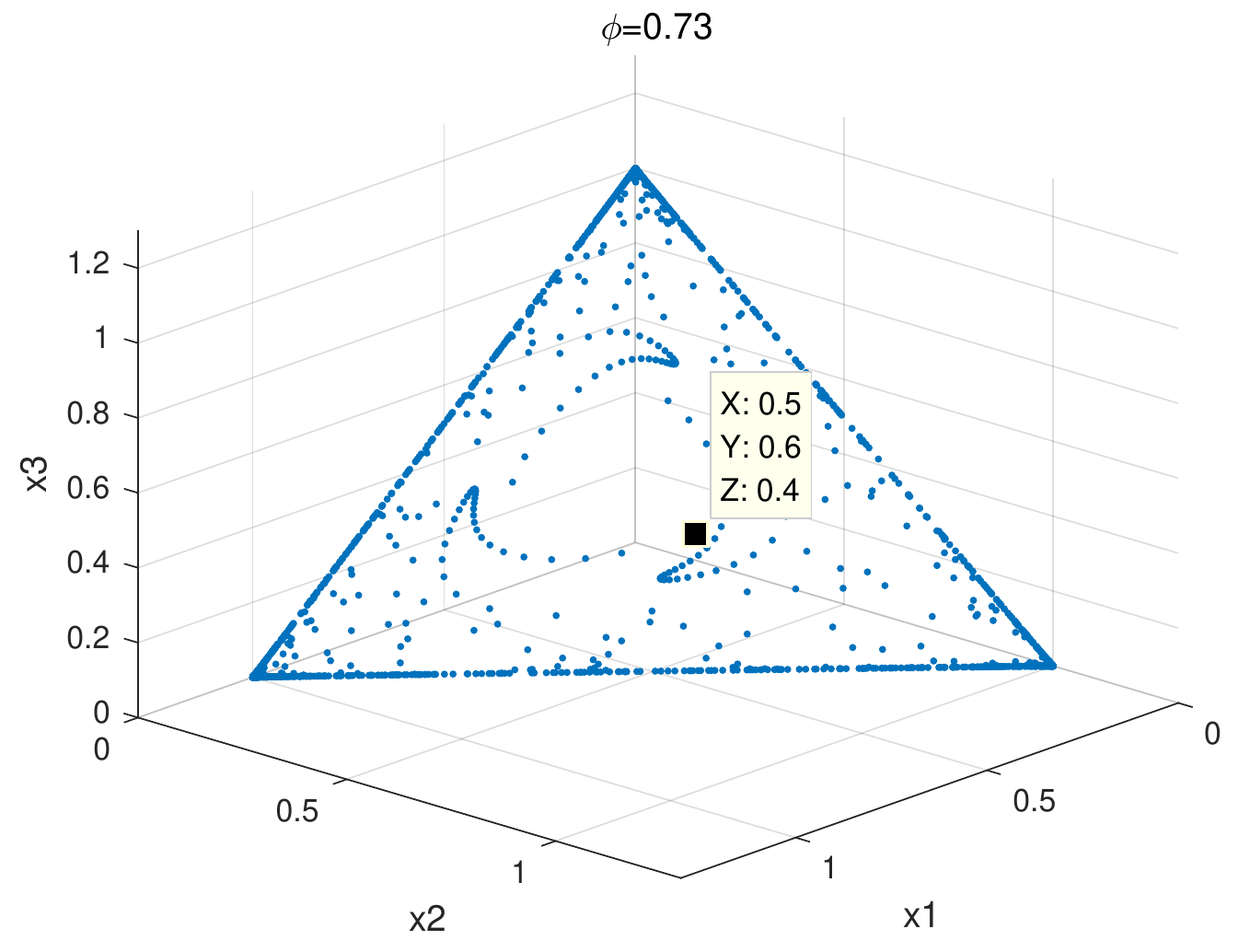}
\caption{\footnotesize Tend to $\tilde{\Gamma}$ }
\end{minipage}
\begin{minipage}[t]{0.48\textwidth}
\centering
\includegraphics[width=6.5cm]{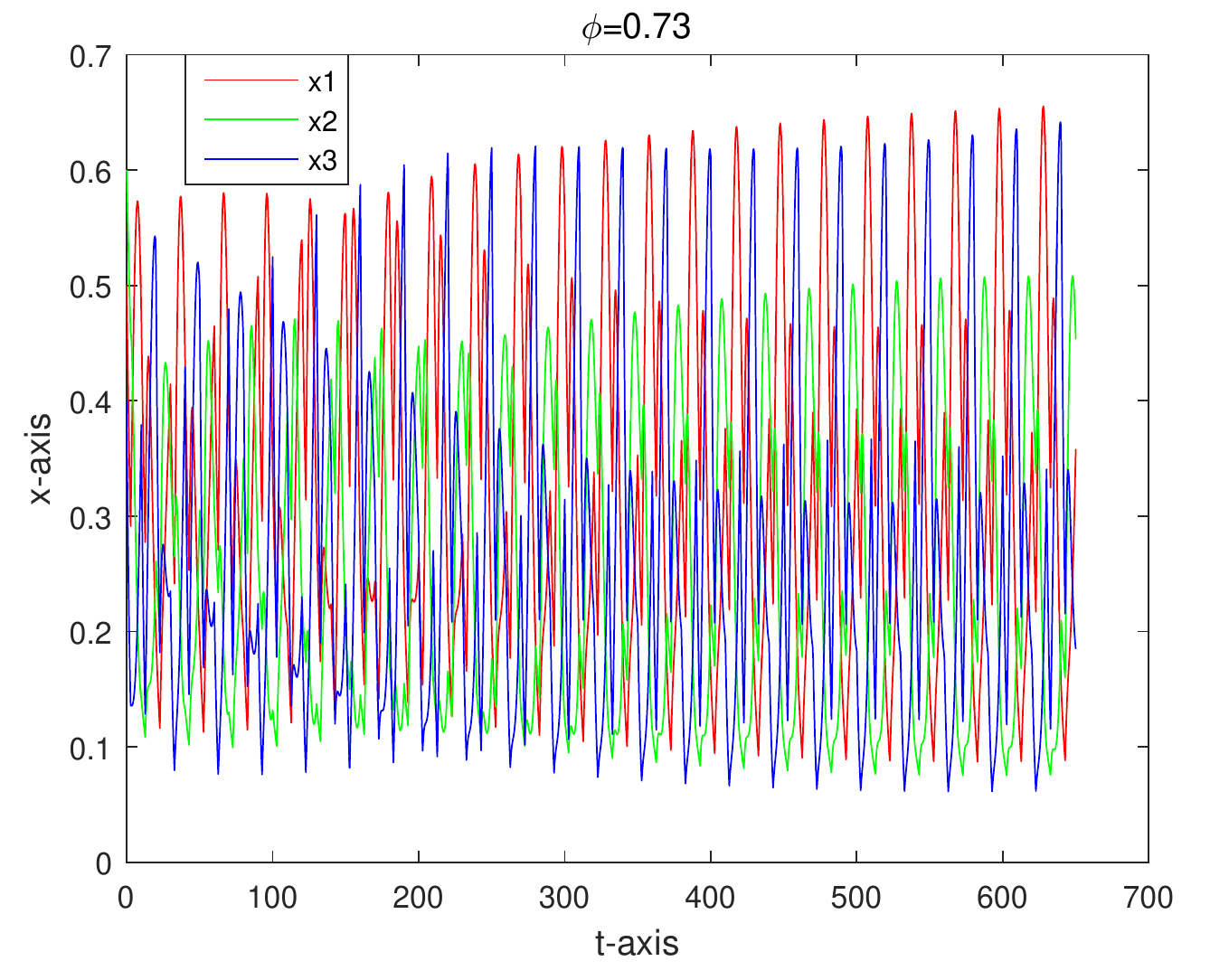}
\caption{\footnotesize Tend to $\tilde{\Gamma}$}
\end{minipage}
\end{figure}
\par When taking $\varphi=0.631$ and $\varphi=0.629 $, one can see that
the orbit emanating from initial value
$x^0=(0.5,0.6,0.4)$ for
$\tilde{S}$
tends to invariant closed curves from Figure 6 and Figure 7. Meanwhile, we also find that
 the invariant closed curve shrinks as $\varphi$ decreases.

\begin{figure}[htbp]
\centering
\begin{minipage}[t]{0.48\textwidth}
\centering
\includegraphics[width=6.5cm]{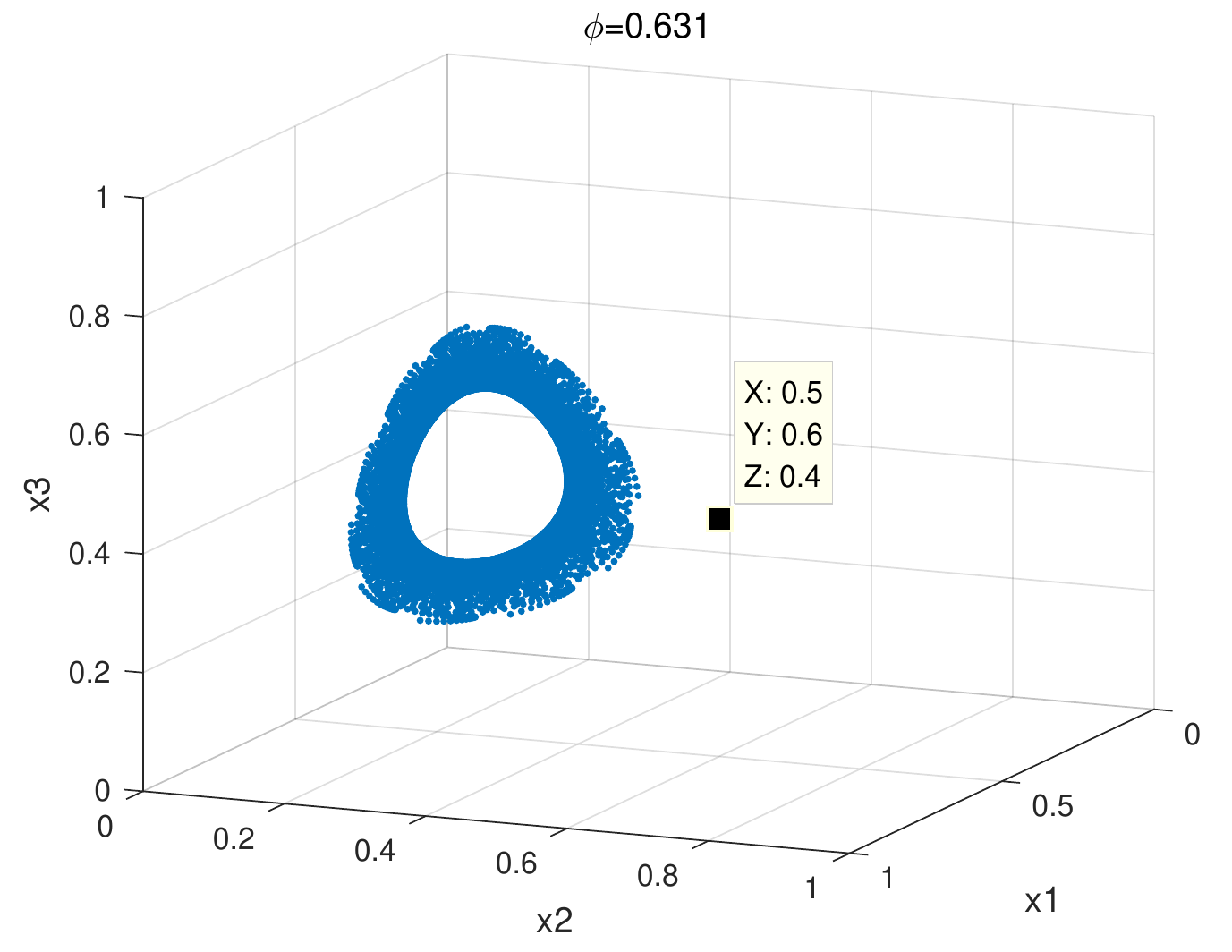}
\caption{\scriptsize Tend to invariant closed curve }
\end{minipage}
\begin{minipage}[t]{0.48\textwidth}
\centering
\includegraphics[width=6.5cm]{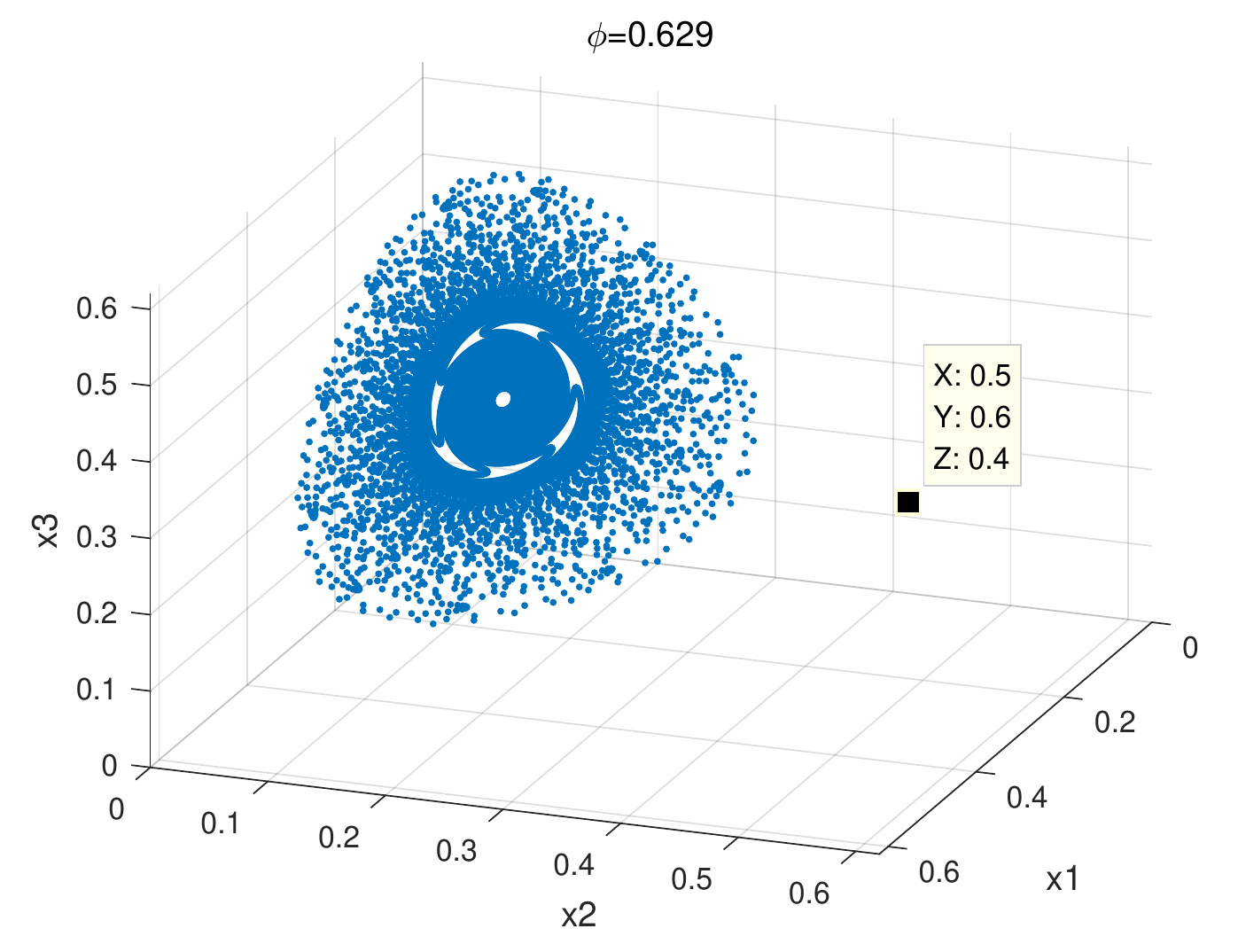}
\caption{\scriptsize Tend to invariant closed curve}
\end{minipage}
\end{figure}

\par When taking $\varphi=0.61$ $\varphi=0.59, \varphi=0.58$ and $\varphi=0.52 $, one can see that
the invariant closed curve for $\tilde{S}$ has degenerated into a positive fixed point from Figure 8-12. Furthermore,
The patterns converging to positive fixed points are rich and varied.

\begin{figure}[htbp]
\centering
\begin{minipage}[t]{0.48\textwidth}
\centering
\includegraphics[width=6.5cm]{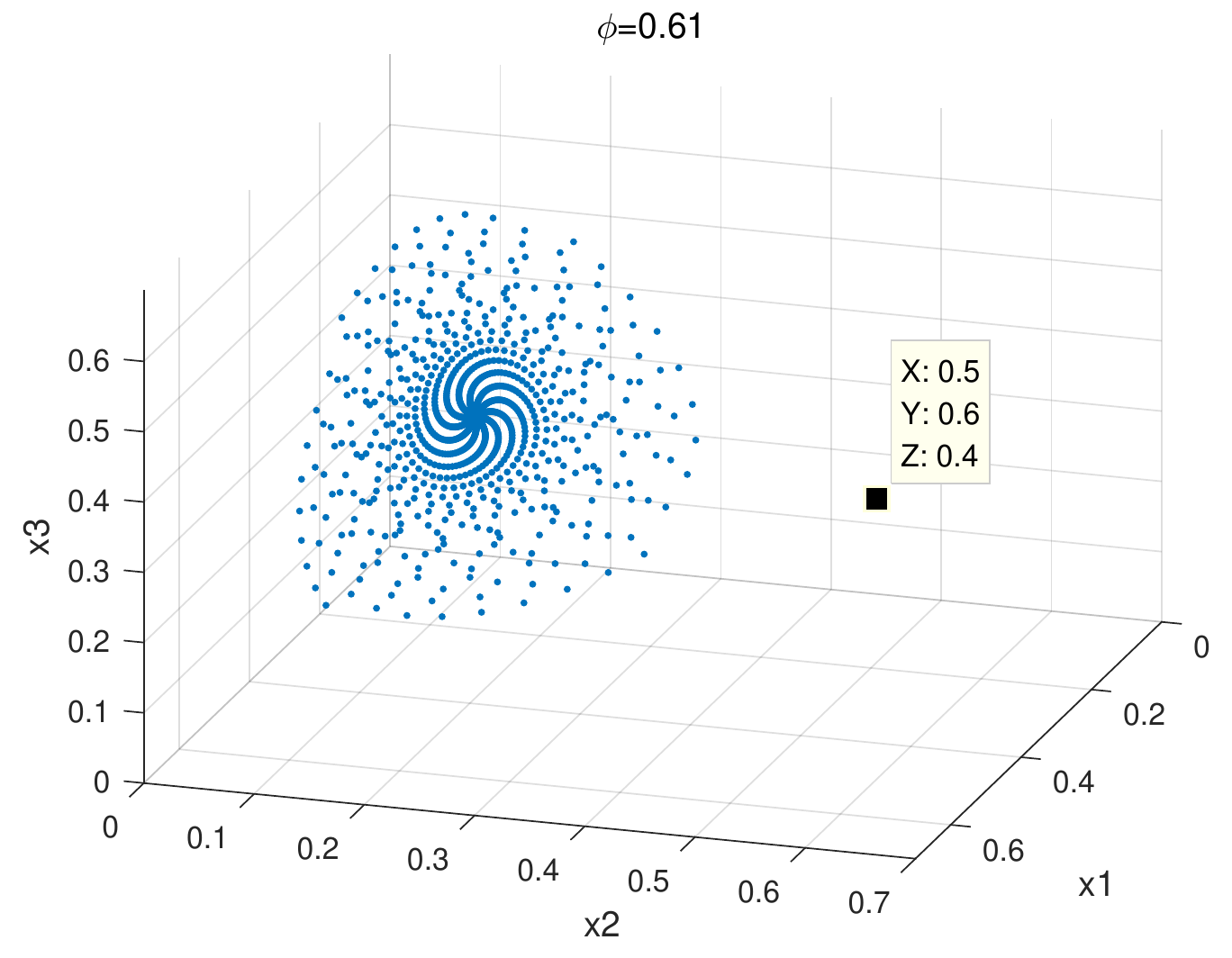}
\caption{\scriptsize Tend to the positive fixed point}
\end{minipage}
\begin{minipage}[t]{0.48\textwidth}
\centering
\includegraphics[width=6.5cm]{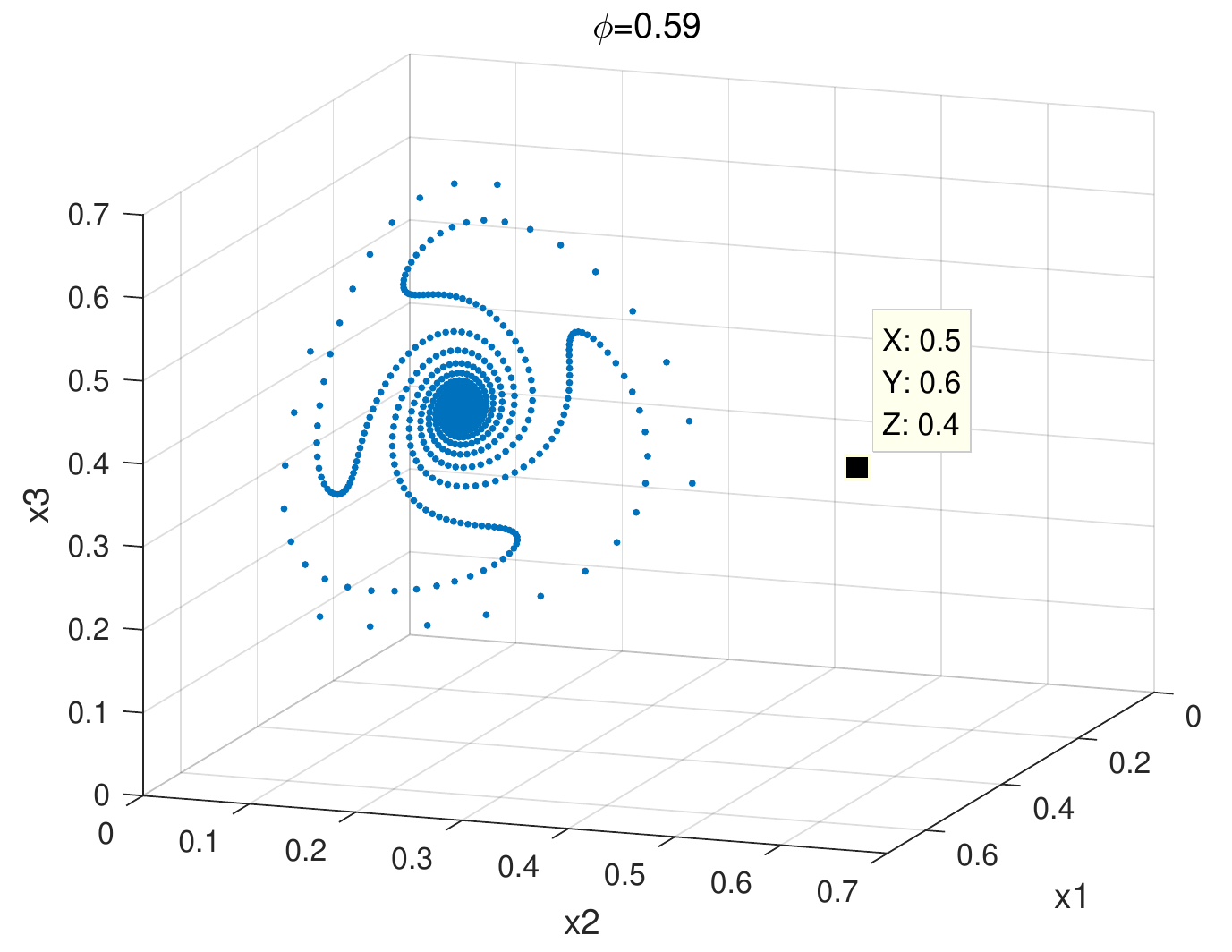}
\caption{\scriptsize Tend to the positive fixed point}
\end{minipage}
\end{figure}

\begin{figure}[htbp]
\centering
\begin{minipage}[t]{0.48\textwidth}
\centering
\includegraphics[width=6.5cm]{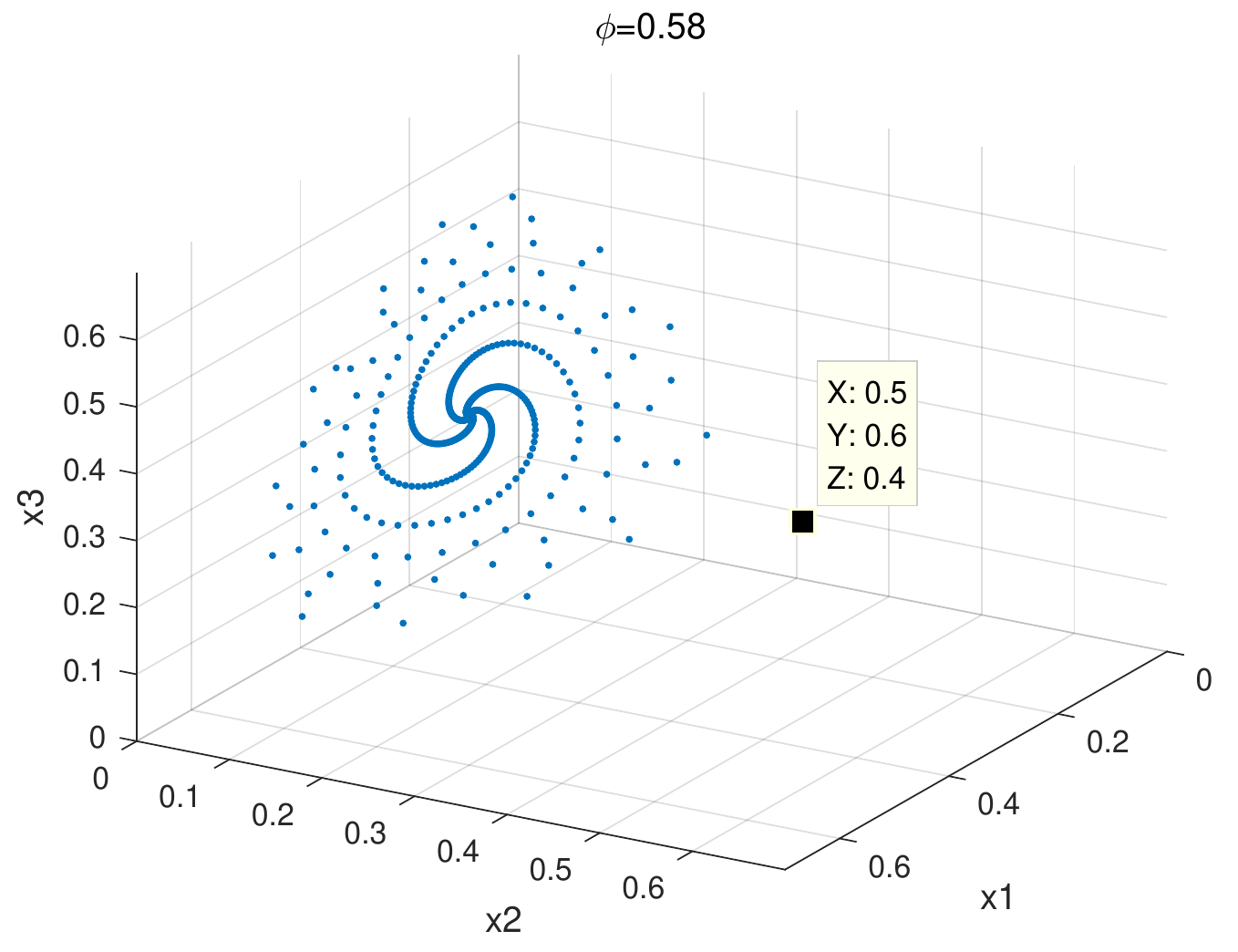}
\caption{\scriptsize Tend to the positive fixed point}
\end{minipage}
\begin{minipage}[t]{0.48\textwidth}
\centering
\includegraphics[width=6.5cm]{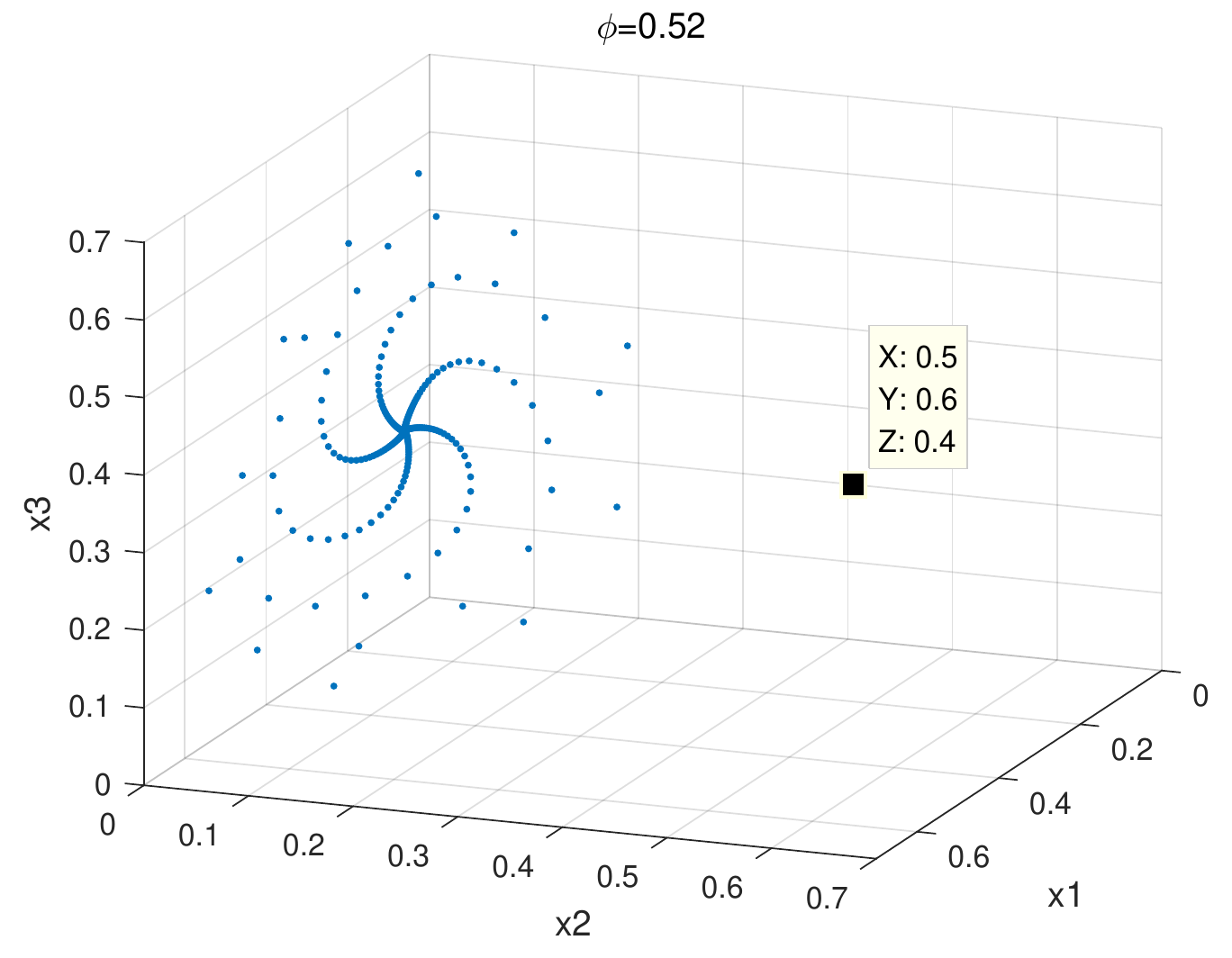}
\caption{\scriptsize Tend to the positive fixed point}
\end{minipage}
\end{figure}
\par By Theorem \ref{thm41}, when $\varphi \in [0,$ $\frac{4.44}{10.54})$, system (\ref{3ML2}) does not admit the heteroclinic cycle.
Using numerical simulation, we find that
the orbit emanating from initial value
$x^0=(0.5,0.6,0.4)$ for
$\tilde{S}$
tends to different fixed points as $\varphi$ decreases gradually. More precisely,
when taking $\varphi=0.39$, one can see that
the orbit for $\tilde{S}$
tends to a positive fixed point from Figure 13;
when taking $\varphi=0.36$, one can see that
the orbit for $\tilde{S}$
tends to the interior fixed point of coordinate plane $\{x_2=0\}$
from Figure 14;
when taking $\varphi=0.34$, one can see that
the orbit for $\tilde{S}$
tends to the interior fixed point of $x_1$-axis
from Figure 15;
when taking $\varphi=0.22$, one can see that
the orbit for $\tilde{S}$
tends to the interior fixed point of coordinate plane $\{x_3=0\}$
from Figure 16;
when taking $\varphi=0.1$, one can see that
the orbit for map $\tilde{S}$
tends to the interior fixed point of $x_2$-axis
from Figure 17;
when taking $\varphi=0.09$, one can see that
the orbit for $\tilde{S}$
tends to the trivial fixed point $O$
from Figure 18.
\begin{figure}[htbp]
\centering
\begin{minipage}[t]{0.48\textwidth}
\centering
\includegraphics[width=6.5cm]{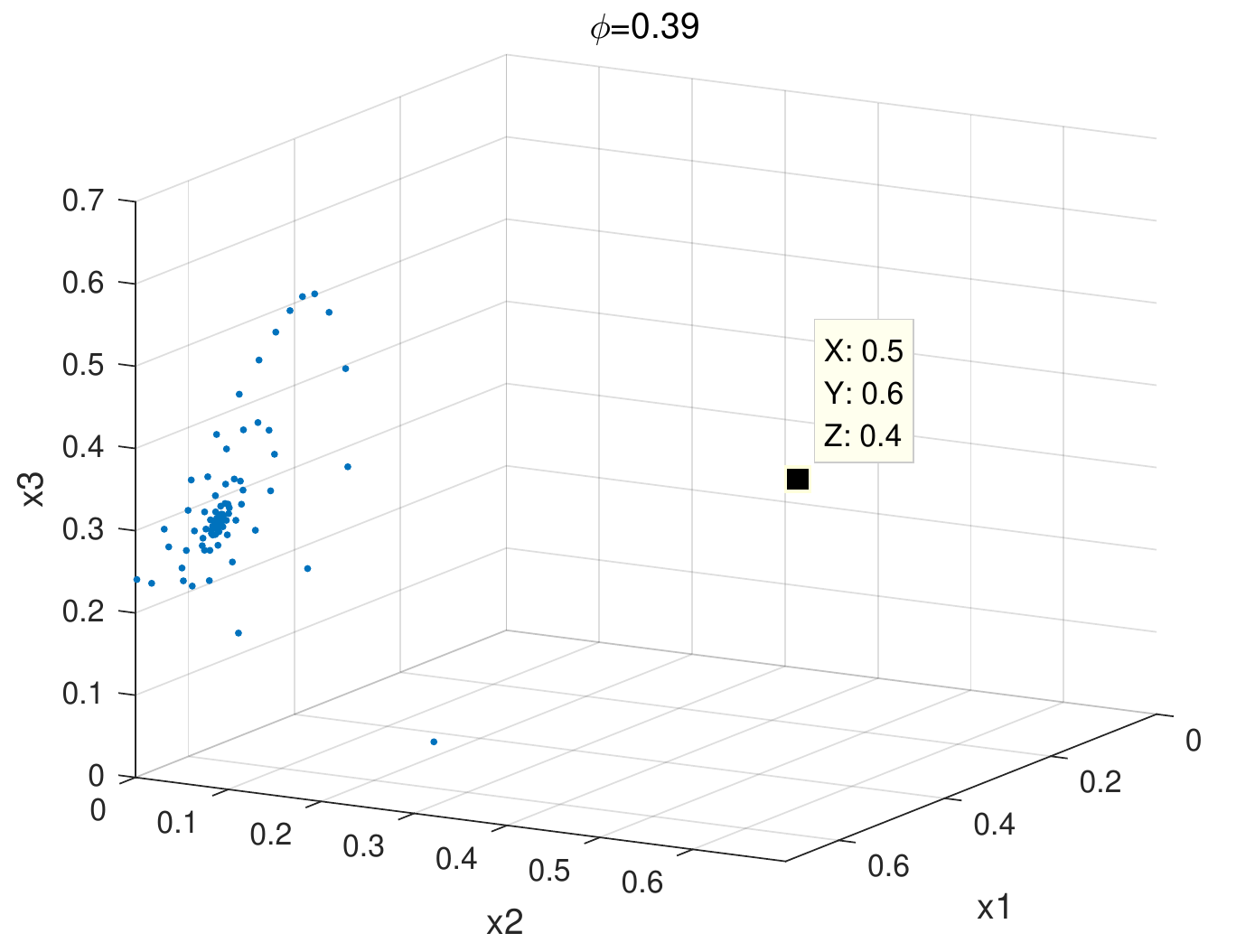}
\caption{\scriptsize Tend to the positive fixed point}
\end{minipage}
\begin{minipage}[t]{0.48\textwidth}
\centering
\includegraphics[width=6.5cm]{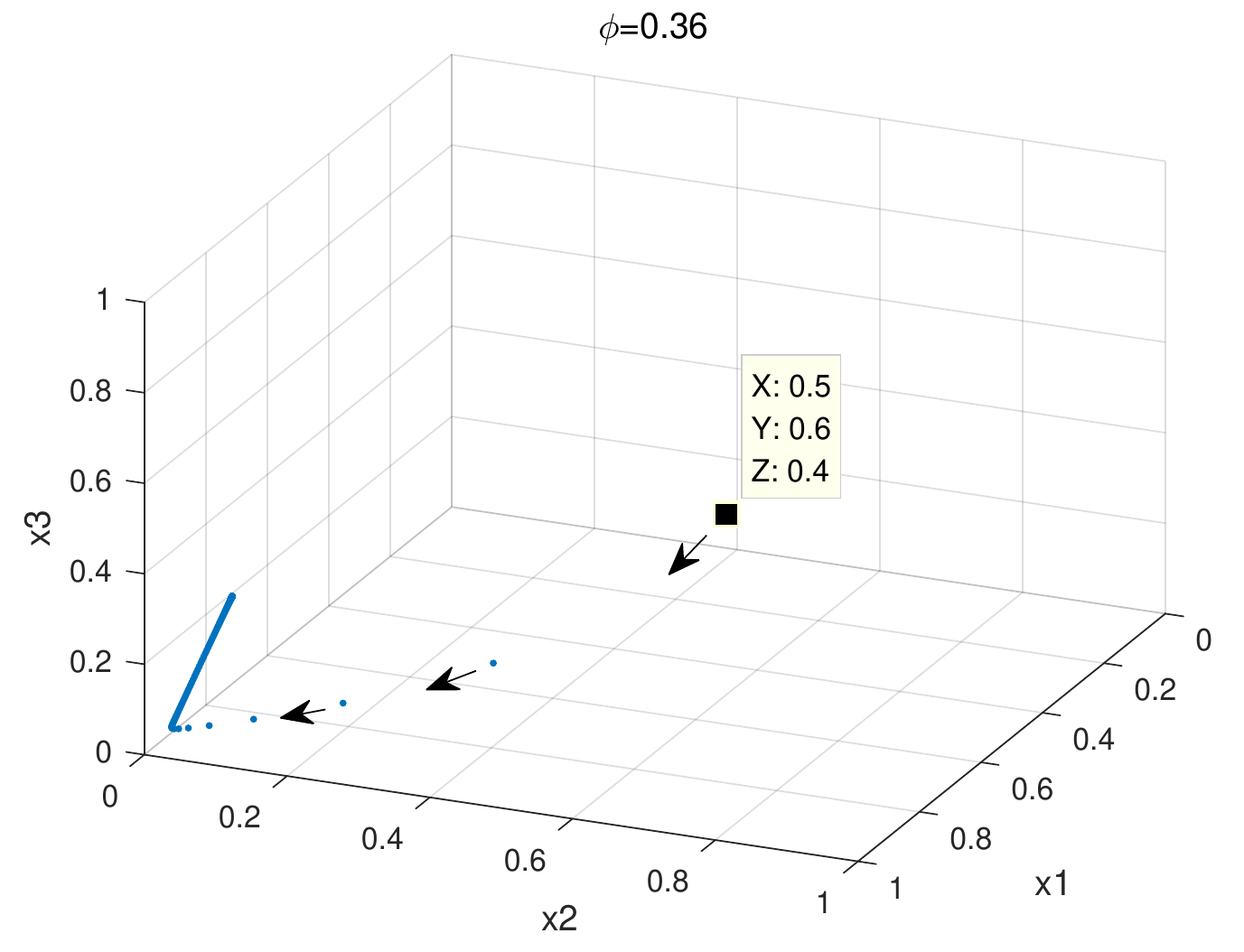}
\caption{\scriptsize Tend to the interior fixed point of $\{x_2=0\}$}
\end{minipage}
\end{figure}

\par In conclusion, compared to the classical May-Leonard competition model (\ref{ML}),
system (\ref{3ML2}) has much richer dynamics
as $\varphi$ varies.
To be more explicit, when $\varphi$ varies in $[0,1]$, system (\ref{3ML2})
enjoys different dynamic features: the orbit emanating from initial value
$x^0=(0.5,0.6,0.4)$ for
the Poincar$\acute{\rm e}$ map $\tilde{S}$
tends to the heteroclinic cycle $\tilde{\Gamma}$(see Figure 4)
$\rightarrow$ invariant closed curves (see Figure 6 and Figure 7)
$\rightarrow$ positive fixed points (see Figure 8-12)
$\rightarrow$ the interior fixed point of coordinate plane $\{x_2=0\}$ (see Figure 13)
$\rightarrow$ the interior fixed point of $x_1$-axis (see Figure 14)
$\rightarrow$ the interior fixed point of the coordinate plane $\{x_3=0\}$ (see Figure 15)
$\rightarrow$ the interior fixed point of $x_2$-axis (see Figure 16)
$\rightarrow$ the trivial fixed point $O$ (see Figure 17).
\begin{figure}[htbp]
\centering
\begin{minipage}[t]{0.48\textwidth}
\centering
\includegraphics[width=6.5cm]{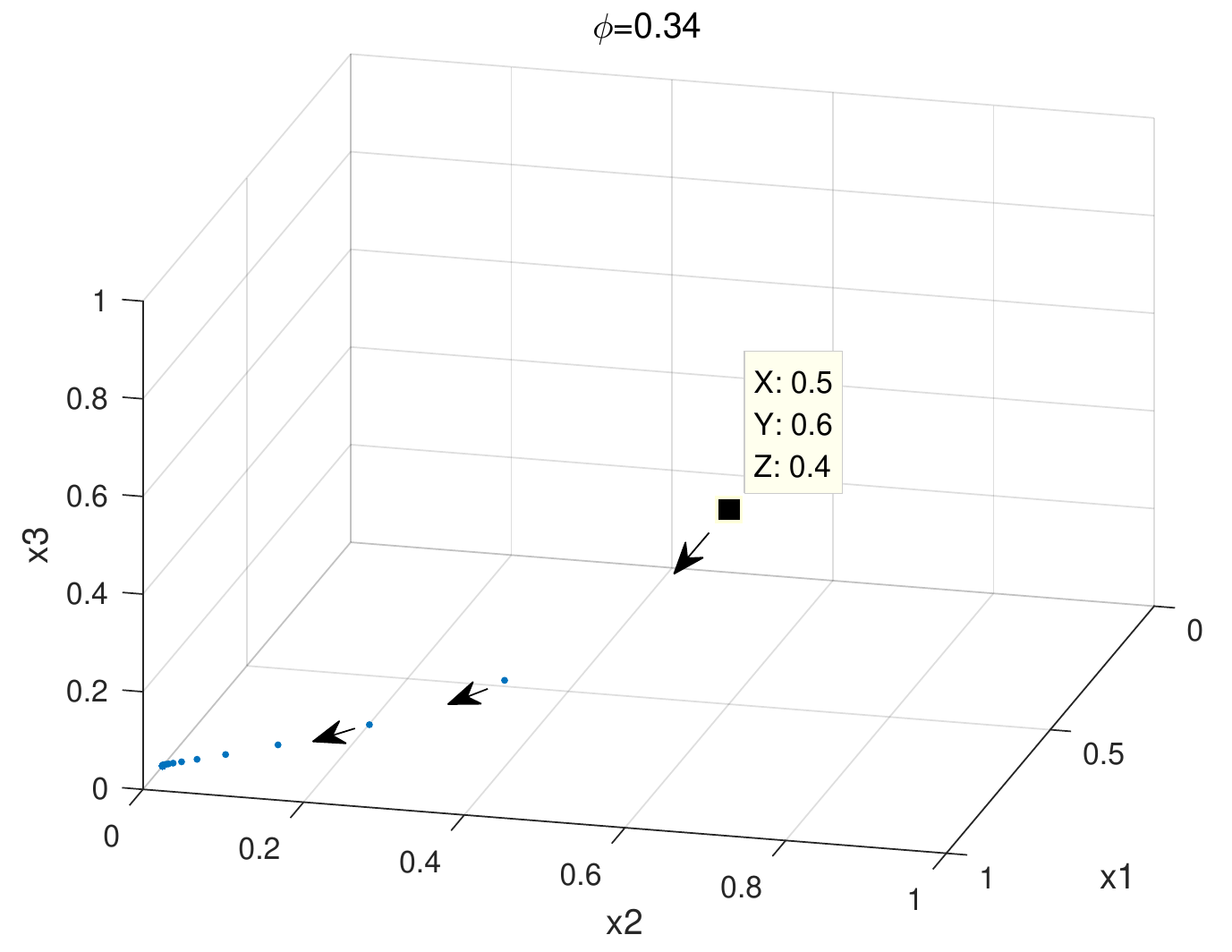}
\caption{\scriptsize Tend to the interior fixed point of $x_1$-axis}
\end{minipage}
\begin{minipage}[t]{0.48\textwidth}
\centering
\includegraphics[width=6.5cm]{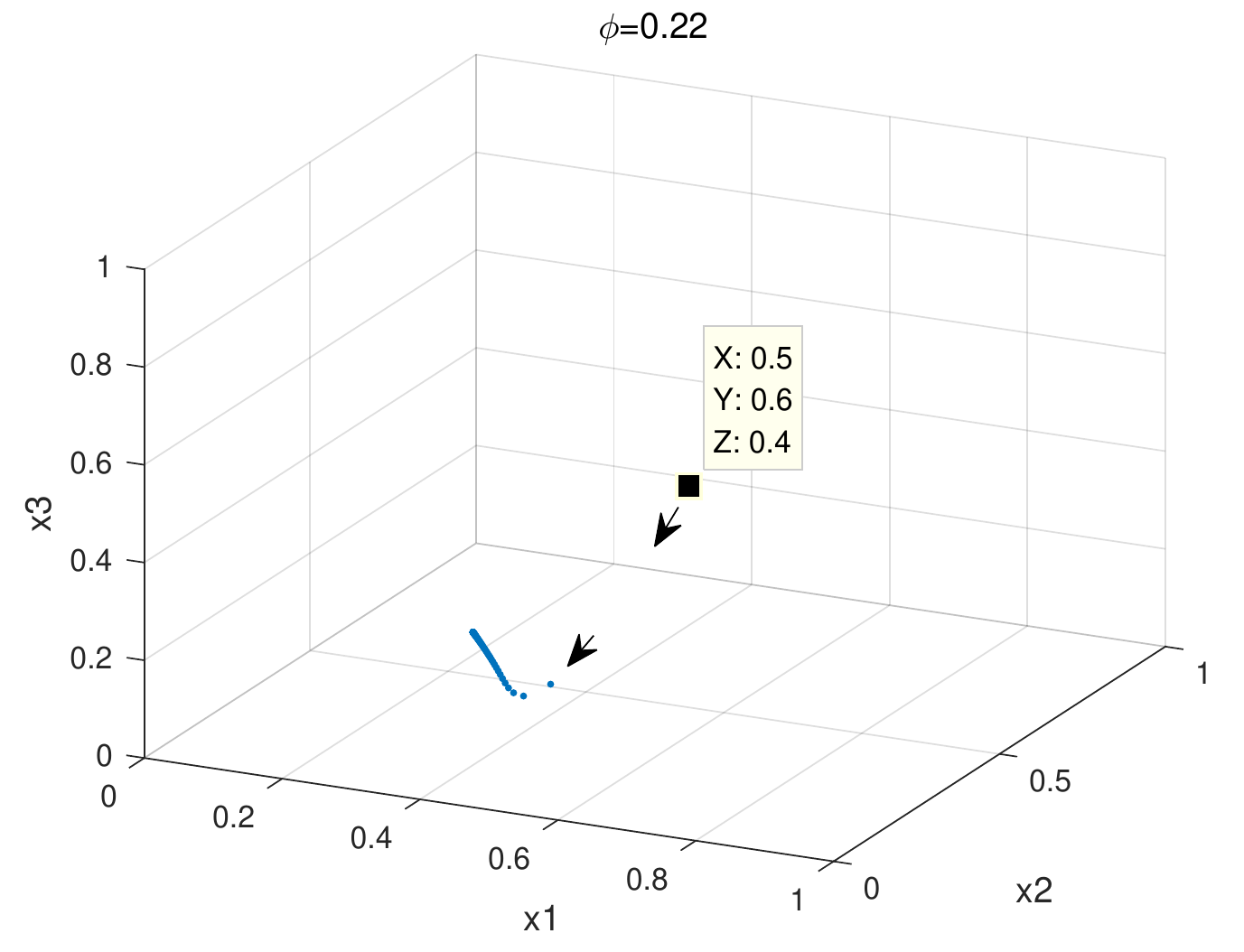}
\caption{\scriptsize Tend to the interior fixed point of $\{x_3=0\}$}
\end{minipage}
\end{figure}

\begin{figure}[htbp]
\centering
\begin{minipage}[t]{0.48\textwidth}
\centering
\includegraphics[width=6.5cm]{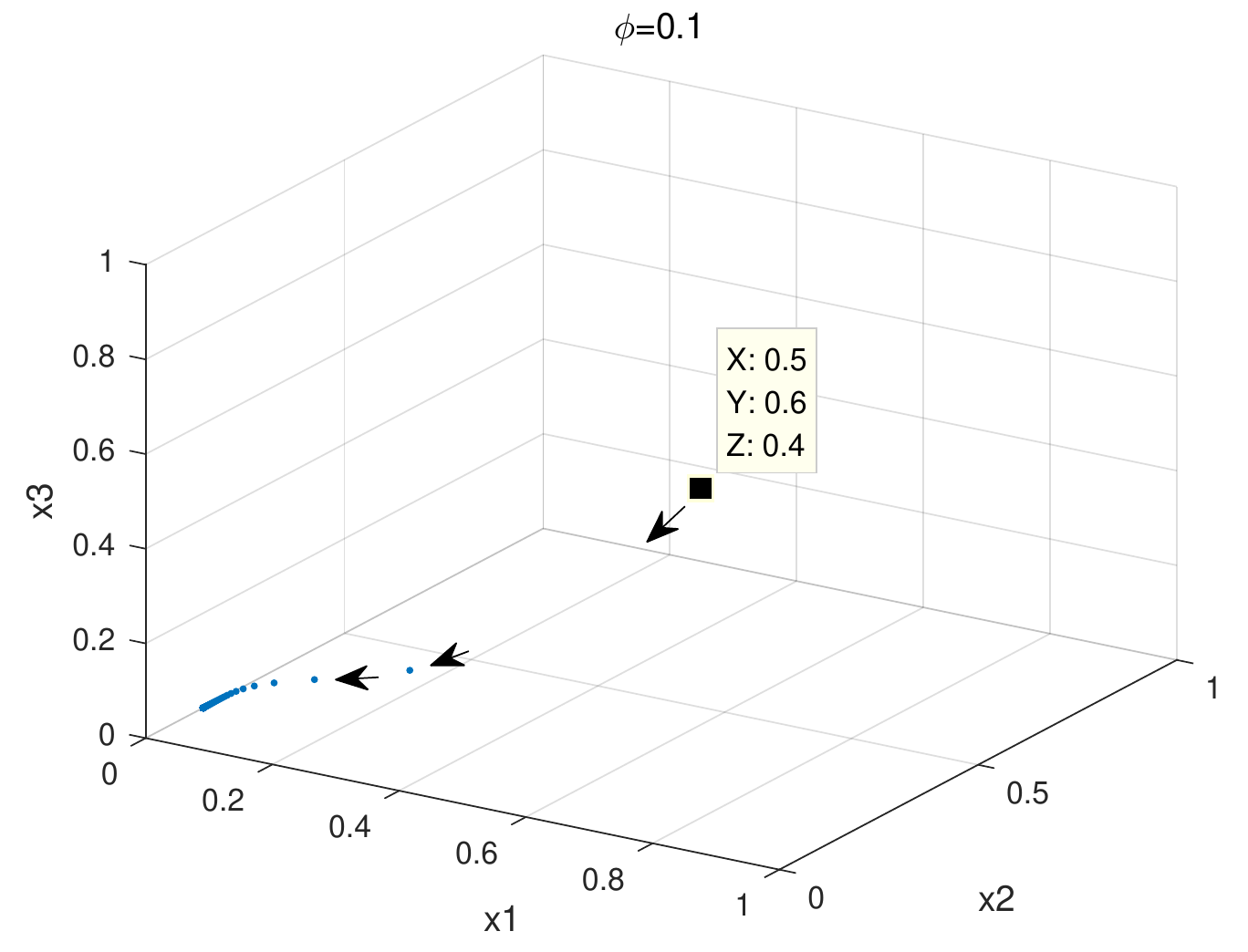}
\caption{\scriptsize Tend to the interior fixed point of $x_2$-axis}
\end{minipage}
\begin{minipage}[t]{0.48\textwidth}
\centering
\includegraphics[width=6.5cm]{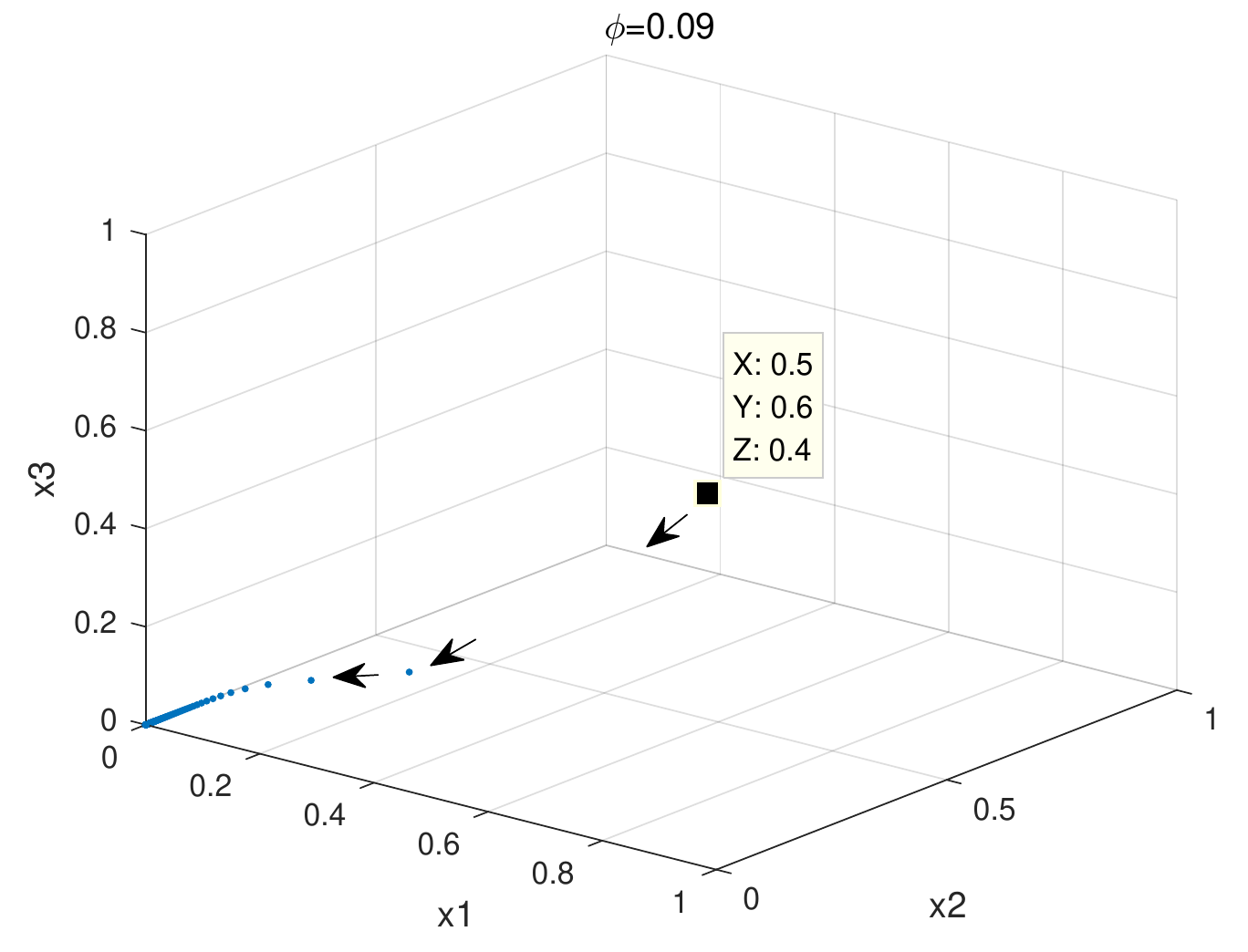}
\caption{\scriptsize Tend to the trivial fixed point $O$}
\end{minipage}
\end{figure}

\section{Discussion}
In this paper, we focus on a symmetric May-Leonard competition model with seasonal succession.
By virtue of the carrying simplex, we obtain that system (\ref{3ML}) admits a heteroclinic cycle
connecting three axial fixed point $R_i, i=1,2,3$, which is the boundary of the carrying simplex.
By the stability theory of heteroclinic cycles for discrete kolmogorov competitive maps
proposed by Jiang et al. \cite{JNW2016},
the stability criterion of the heteroclinic cycle for system (\ref{3ML}) is established.
Especially, when $\lambda_1=\lambda_2=\lambda_3$ and $\alpha+\beta \neq 2$, it is shown that
system (\ref{3ML}) inherits stability properties of the heteroclinic cycle of the classical May-Leonard
competition model (\ref{ML}), which implies that the stability of the heteroclinic cycle has not been changed by
the introduction of seasonal succession in this case.
Besides, we also present the explicit expression of the carrying simplex, which states that
the carrying simplex is flat in the special case. Applying the stability criterion to
a given system (\ref{3ML2}), we obtain that the stability bifurcation of the heteroclinic cycle
for such a system. By numerical simulation, we not only illustrate the effectiveness of our
theoretical results, but also exhibit rich dynamics (including
nontrivial periodic solutions, invariant closed curves and heteroclinic cycles)
for the system. Compared to concrete discrete-time competitive maps discussed in
\cite{GJNY2018,GJNY2020,JL2016,JL2017,JN22017}, the dynamic patterns converging to
invariant closed curves and positive fixed points for system (\ref{3ML2}) are much richer.
We guess that the interesting and beautiful dynamic patterns are generated by the switching between two subsystems.
\par In addition, lots of numerical simulations strongly suggest that if
system (\ref{3ML}) has one positive periodic solution, then the positive periodic solution is unique.
The estimate for the Floquet multipliers of the positive
periodic solution and the complete classification for the global dynamics of system (\ref{3ML})
are all challenging problems for us. We will leave them for future research.

%\bibliographystyle{/home/dylan/Documents/Project/bb/siamplain}
%\bibliography{/home/dylan/Documents/Project/bb/dd.bib}
\bibliography{D:/Documents/Projects/bb/dd}
\bibliographystyle{D:/Documents/Projects/bb/siamplain}

\end{document}